\documentclass[a4paper, 12pt, reqno]{amsart} %
\allowdisplaybreaks
\usepackage{amsmath,amssymb,amsthm}
\usepackage{microtype}
\usepackage[left=2.5cm, right=2.5cm, top=3cm, bottom=3cm]{geometry} 
\usepackage[colorlinks = true, urlcolor = blue, linkcolor = blue,
            citecolor = blue]{hyperref}
\usepackage{tikz} 
\usetikzlibrary{calc}
\usetikzlibrary{graphs}
\usetikzlibrary{graphs.standard}
\usetikzlibrary{arrows.meta,bending}
\newcommand{\mycycle}[2]{
  \def\numpoly{#1}    
  \def\startangle{90} 
  \def\pradious{#2}   
  \pgfmathparse{int(\startangle+360/\numpoly)}%
     \let\nextangle=\pgfmathresult
  \pgfmathparse{int(\startangle-360/\numpoly+360)}%
    \let\endangle=\pgfmathresult
  \foreach \i [count=\ii from 1] in
  {\startangle,\nextangle,...,\endangle} {
    \pgfmathsetmacro{\mycolor}{\LstColors[\ii-1]}
    \path (\i:\pradious) node[fill=\mycolor, draw=black, shape=circle,
    inner sep=1pt, minimum size=2.5mm, thick] (p\ii) {};
  }
  \foreach \i [count=\ii] in {2,...,\numpoly, 1} { 
    \draw[thick, black] (p\ii) -- (p\i);
  }
}
\usepackage[shortlabels]{enumitem}

\usepackage{bm}
\newlist{myenum}{enumerate}{2}
\setlist[myenum,1]{label*=$(C_{\arabic*})$.} 
\setlist[myenum,2]{label=$(C_{\arabic{myenumi}}).\alph*$.}
\usepackage{xcolor}
\definecolor{rrorange}{HTML}{FF8300}            
\definecolor{rrgreen}{HTML}{299617}             
\definecolor{rrblue}{HTML}{1F75FE}              
\definecolor{scarlet}{rgb}{1.0,0.13,0.0}        %
\definecolor{rrbrown}{rgb}{0.75,0.25,0}         
\definecolor{sienna}{rgb}{0.53, 0.18, 0.09}
\definecolor{tenne}{rgb}{0.8, 0.34, 0.0}
\definecolor{Cred}{HTML}{FF0000} %
\definecolor{Cgray}{HTML}{0FAFFF} %
\definecolor{Corange}{HTML}{FF8C00} %
\definecolor{Cgreen}{HTML}{008000} %
\newtheorem{theorem}{Theorem}
\newtheorem{lemma}{Lemma}
\newtheorem{corollary}{Corollary}

\theoremstyle{definition}
\newtheorem{definition}{Definition}

\newtheorem{remark}{Remark}

\DeclareMathAlphabet{\txcal}{U}{tx-cal}{m}{n}
\DeclareMathAlphabet{\ducal}{U}{dutchcal}{m}{n}
\DeclareMathAlphabet{\rsocal}{U}{rsfso}{m}{n}
\usepackage{mathrsfs}
\newcommand{\Ber}{\ducal B}
\newcommand{\Bhood}{\mathcal B}

\newcommand{\CR}{\txcal R(\Phi)}
\newcommand{\D}{\txcal D}
\renewcommand{\det}{\operatorname{\textup{det}}}
\newcommand{\E}{\mathbb{E}}

\newcommand{\G}{\rsocal G}
\newcommand{\eS}{\mathcal{S}}
\newcommand{\eC}{\mathcal{C}}
\newcommand{\eK}{\mathcal{K}} 
\newcommand{\F}{\mathfrak{F}}
\newcommand{\Fix}{\mathsf{Fix}}
\newcommand{\iD}{\D^\circ}
\newcommand{\iG}{\mathit{\Xi}}
\newcommand{\Ind}{{\mathbf 1}}
\newcommand{\JacF}[1]{\mathit{DF}(#1)}

\newcommand{\limset}{\mathfrak{L}(X)}
\newcommand{\N}{\mathbb N}
\newcommand{\Prb}{\mathbb P}
\DeclareMathOperator*{\powerset}{\raisebox{.15\baselineskip}{\Large\ensuremath{\wp}}}
\newcommand{\quasegrd}{\mathfrak g}
\newcommand{\R}{\mathbb R}
\newcommand{\rpar}{\rho}

\newcommand{\Supp}[1]{\operatorname{\textup{supp}}(#1)} 
\newcommand{\TsD}{\mathrm{T}\D}
\newcommand{\Z}{\mathbb Z}
\newcommand{\Var}{\textup{Var}}
\newcommand{\V}{\mathfrak{V}} 
\usepackage{mathtools}

\DeclarePairedDelimiter\floor{\lfloor}{\rfloor}

\newcommand{\dbr}[1]{\left\langle\!\left\langle
      #1\right\rangle\!\right\rangle}

\begin{document}
\title[Interacting vertex reinforced random walks]{Interacting vertex
  reinforced random walks on complete sub-graphs}
\author[F. P. A. Prado]{Fernado P. A. Prado}
\address[F. P. A. Prado and
         R. A. Rosales]{Departameto de Computa\c{c}\~ao e
           Matem\'atica, Universidade de S\~ao Paulo, 
	   Avenida Bandeirantes 3900, Ribeir\~ao Preto, S\~ao Paulo, 
	   14040-901, Brasil}
\email{feprado@usp.br, rrosales@usp.br}
\author[R. A. Rosales]{Rafael A. Rosales}

\date{\today} 

\begin{abstract}
This article introduces a model for interacting vertex-reinforced
random walks, each taking values on a complete subgraph of a locally
finite undirected graph. The transition probability for a walk to a
given vertex depends on the cumulative proportion of visits by all
walks that have access to that vertex. Proportions are modified by
multiplication by a real valued interaction parameter and the addition
of a parameter representing the intrinsic preference of the walk for
the vertex. This model covers a wide range of interactions, including
the cooperation (attraction) or competition (repulsion) of several
walks at single vertices.  We are principally concerned with strong
laws for the proportion of visits to each vertex by all walks.  We
prove that this measure converges almost surely towards the set of
fixed points of the transition probabilities. Almost sure convergence
to a single fixed point is in fact the generic behaviour as we show
this to hold for almost all parameter values of our model.

Beyond almost sure convergence, our model provides a general framework
that yields a detailed description of the limiting behaviour for any
choice of interaction parameters and subgraph geometry. We illustrate
this by analyzing interacting walks on complete graphs, stars, and
cycles, chosen to highlight the model's broad applicability. The
central contribution lies in offering a powerful tool to analyse
diverse types of interactions mediated by the intersections of
subgraphs. Importantly, our results provide not only convergence
criteria, but also conditions under which the empirical proportions of
the walks' visits fail to converge to certain boundary points of the
space where their trajectory evolves.
\end{abstract}

\maketitle

\section{Introduction}\label{sec:intro}
Self vertex-reinforced random walks on a connected locally finite
graph, $G = (V, E)$, have attracted a great deal of attention for more
than three decades. This process was first formally described by
\cite{P92} as a variation of the edge-reinforced walk introduced in
\cite{CD87} and then subsequently studied in \cite{B97}. The evolution
along the vertices of $G$ is governed by a transition probability
which is a function of the number of times that the walk has visited
each vertex. When the weight function is non-decreasing, the
transition probabilities tend to favour vertices that have already
been visited by the walk.  A distinctive property of this process is
that the walk can become confined to a proper subgraph of $G$: for
all sufficiently large numbers of steps the trajectory of the walk is
restricted to a random but fixed subset of the vertex set $V$. This
behaviour, referred to as localisation, has been extensively studied
and shown to occur almost surely or with positive probability for
several weight functions and a relatively general class of graphs, see
\cite{B97}, \cite{PV99}, \cite{V01}, \cite{T04}, \cite{BT11},
\cite{BSS14}. \cite{CT2017} show that if the weight function increases
sufficiently fast, then the walk localises almost surely on two
vertices in a large class of graphs. It is conjectured that the size
of the attracting subgraph depends on the geometry of $G$, but this
question is in general a widely open problem.

In contrast to a large literature on self vertex-reinforced random
walks, studies addressing their interacting counterparts are
relatively scarce. The articles \cite{C14}, \cite{CDLM19},
\cite{RPP22} and \cite{PCR2023} present some steps in this direction,
but there is substantial room for further development. \cite{C14}
considers a model of two vertex-reinforced random walks on finite
complete graphs in which the transition probability of one of the
walks decreases with the number of visits to that vertex by the other
walk. This article shows that the overlap of empirical vertex
occupation measures by each walk is in the long run almost surely
arbitrarily small. \cite{CDLM19} consider a model of $m>1$ interacting
vertex-reinforced random walks and focus on their synchronisation
towards a common limit, building upon the framework
described in \cite{DLM14}. A functional central limit for the
fluctuations about this limit is also presented. \cite{PCR2023}
consider a model for two interacting random walks on the two vertex
graph as a device to study the recurrence properties of two
exponentially repelling random walks taking values on
$\Z$. \cite{RPP22} generalises \cite{PCR2023} by considering $m>1$
interacting walks with an exponential weight function on complete
graphs. The article demonstrates that the vertex occupation measure of
each walk converges almost surely, provided the limit set consists of
isolated points.  Besides these works, \cite{ER24} and \cite{GMR2024}
have recently considered two distinct models of interacting
edge-reinforced random walks. \cite{ER24} describes a process on
directed graphs, while \cite{GMR2024} considers interaction on $\Z$
and on a linear three vertex graph. Both articles obtain localisation
results.

This article considers a model of interacting vertex-reinforced random
walks in which each walk takes values on a complete subgraph of $G$
and interacts with others through shared vertices.  Our model arises
as a natural generalisation of the self reinforced random walk in
\cite{P92} to include interactions.  To introduce the model, let $i
\in [m] = \{1, 2, \ldots, m\}$ and $V^i \subset V$ be such that $V =
\bigcup_{i=1}^m V^i$. Define $G^i =(E^i,V^i)$, $i \in [m]$, as the
complete graph with vertices $V^i$ and denote by $\G$ the set of
complete graphs $\G =\{G^i = (E^i, V^i); i \in [m]\}$. For each $i\in
[m]$ let $W^i = \{W^i(n);\  n \geq 0\}$ be a random walk taking values
on $V^i$. These walks are defined on a suitable probability space
$(\Omega, \F, \Prb)$ as follows.  For each $i\in[m]$, $v \in V^i$,
set $X_v^i(0) = 1/d^i$ where $d^i = |V^i|$.  For $n\geq 1$ denote by
$X_v^i(n)$ the empirical occupation measure of the vertices $v \in
V^i$ by the $i$-th walk, that is
\begin{equation}
  \label{eqn:occupation_measure}
  X_v^i(n)
  = \frac{1}{d^i + n}\bigg(1 + \sum_{k=1}^n \Ind\big\{W^i(k)
  =  v\big\}\bigg).
\end{equation}
This is of course also the relative local time of $W^i$ at $v$ up to
time $n$, if one assumes that $X_v^i(0) = 1/d^i$ is the relative local
time built up over the $d^i$ steps before time zero by the walk $W^i$. 

Let
\begin{equation}\label{eqn:X}
  X = \{X(n);\ n\geq 0\} \quad \text{with} \quad
  X(n) = (X^1(n), X^2(n),\ldots, X^m(n))
\end{equation}
be the vertex occupation process, where
\[
  X^i(n) = (X^i_1(n), \ldots, X^i_{d^i}(n)), \qquad i\in [m]. 
\]
When necessary, we shall order the coordinates of $X^i(n)$ as
$\big(X_{\underline{1}}^i(n)$, $\ldots$,
$X_{{\underline{d}}^i}^i(n)\big)$, where, for each $i\in [m]$ and
$\ell \in [d^i]$, the vertex $\underline{\ell}^i \in V^i$ is
recursively defined as $\underline{1}^i = \min\{V^i\}$ and
$\underline{\ell}^i = \min\big\{V^i{\setminus}\{\underline{1}^i,
\ldots, \underline{(\ell-1)}^i\}\big\}$ for $\ell > 1$.  To simplify
notation, hereafter we drop any super index of underlined sub indexes
in double indexations. That is, $x_{\underline{\ell}}^i$ will be used
instead of $x_{\underline{\ell}^i}^i$. Observe that $X$ takes values
on the compact convex set $\D = \triangle^1
\times\cdots\times\triangle^m$; where $\triangle^i$ is the unitary
$d^i-1$ simplex, that is
\[
  \triangle^i
  =
  \bigg\{
    (x_1, \ldots, x_{d^i})
    \in \R^{d^i}\ \Big|\  
    x_{\ell} \geq 0,\ \ell=1,2,\ldots, d^i\ \text{and } 
    \sum_{\ell = 1}^{d^i} x_{\ell} = 1 
  \bigg\}.
\]

Let $\F_0 = \{\Omega, \varnothing\}$ and denote by $\F_n$ the
filtration generated by the walks $W^i$, $i\in [m]$, up to time $n
\geq 1$. Further, let 
\begin{equation}\label{eqn:Iv}
  I_v = \big\{i \in [m]\ \big|\ v \in V^i\big\}
\end{equation}
be the set of indices identifying those walks that have the vertex $v$
in common.  Conditionally on $\F_n$, for any $i \in [m]$, $v \in V^i$
and $n\geq 1$, the transition probability for the walk $W^i$ is given
by
\begin{equation} 
 \label{eqn:trans_prob}
  \Prb\big(W^i(n+1)= v \mid\F_n\big) = \pi_v^i(X(n)) 
\end{equation}
for $\pi_v^i : \D \to \R$, $i \in [m]$, $v \in V^i$, as the map
\begin{equation}
  \label{eqn:pi}
  \pi^i_v (x) = 
  \frac{x_v^i H_v^i(x)}{\sum_{w \in V^i} x_w^i H_w^i(x)},
  \qquad
  H_v^i(x)
  =
  \bigg(\eta_v^i + \sum_{j \in I_v} \rpar_{v}^{ij} x_v^j\bigg)^\alpha, 
\end{equation}
where 
\begin{equation}\label{eqn:modelpar}
  \eta_{v}^i > 0, \quad \alpha > 0, \quad \rpar_{v}^{ij} \in \R \quad
  \text{ for } \quad v \in V \, \text{ and } \, i, j \in I_v 
\end{equation}
are parameters of the model. Throughout this article, $\rpar_v^{ij}$
and  $\eta_v^i$ are subjected to the conditions
\begin{gather}
 \label{eqn:rho}
 \rpar_v^{ij} = \rpar_v^{ji} \\ 
 \label{eqn:eta}
 \eta_v^i > \sum_{j \in I_v: \rpar_{v}^{ij} < 0}
 \big|\rpar_{v}^{ij}\big|.
\end{gather}
The latter condition ensures that the numerator of $\pi_v^i$ is
non-negative and the denominator is greater than zero. We assume
further the random variables $W^1(n + 1)$, $\ldots$, $W^m(n + 1)$ to
be independent conditionally on $\F_n$. The joint process $\{W(n) =
(W^1(n), \ldots, W^m(n)); n\geq 0\}$ is completely defined by
specifying the initial condition $X^i_v(0)$ and the smooth map $\pi:
\D\to\D$, which at $x=\big(x^1_{\underline{1}}$, $\ldots$,
$x^1_{\underline{d}^1}$, $\ldots$, $x^m_{\underline{1}}$, $\ldots$,
$x^m_{\underline{d}^m}\big)$ takes the value
\begin{equation}
  \label{eqn:THE_pi} \pi(x) = \big(\pi^1_{\underline{1}}(x), \ldots,
\pi^1_{{\underline{d}}^1}(x), \ldots, \pi^m_{\underline{1}}(x),
\ldots, \pi^m_{{\underline{d}}^m}(x)\big).
\end{equation}

As we shall see, this model covers a wide variety of interactions,
which include different strengths of vertex repulsion and attraction
determined by the interaction parameters $\rpar_v^{ij}$. Note that the
factor $H_v^i(x)$ increases with the parameter $\eta_v^i$, which
represents the intrinsic preference of walk $i$ for vertex
$v$. Furthermore, $H_v^i(x)$ also depends on $\rpar_{v}^{ij} x_v^j$
with $j \in I_v$, indicating the weighted proportions of visits to $v$
by all other walks $j$ that may visit $v$. The model is quite general,
encompassing scenarios such as $\rpar_{v}^{ij} < 0 < \rpar_{w}^{ij}$
for some $v, w, i,$ and $j$. In this case, walks $i$ and $j$ repel
each other at vertex $v$ but attract each other at vertex
$w$. Additionally, it may be the case that $\rpar_{v}^{ii} < 0 <
\rpar_{w}^{ii}$, where walk $i$ is negatively influenced by its own
visits to vertex $v$ but positively influenced with respect to its
visits to vertex $w$.

Determining whether the vertex occupation process $X$ converges almost
surely and identifying its possible limits in $\D$ depending upon
$\G$, $\rpar_v^{ij}$, and $\eta_v^i$, are non trivial questions. This
article aims at addressing these issues in some generality. We show
first that the almost sure asymptotic behaviour of $X$ is intimately
related to the set of fixed points of $\pi$. We then provide a general
characterisation of this set and under a mild regularity condition we
show that $X$ converges almost surely towards some element of this
set. The almost sure convergence of $X$ is in fact the generic
behaviour as this is shown to occur for almost all parameters
satisfying \eqref{eqn:rho} and \eqref{eqn:eta}.  This represents a
significant advance over \cite{C14}, \cite{RPP22} and \cite{PCR2023},
where convergence is established only for the case of two interacting
walks on a two vertex graph, or on complete graphs under the
additional assumption that the set of possible limit points is
finite. Far beyond the mere convergence of $X$, the model just
introduced allows for a detailed description of the limit points of
$X$ depending upon $\G$ and the parameters $\rpar_v^{ij}$ and
$\eta_v^i$. This characterisation relies in part on non-convergence
results of $X$ to specific points in the relative boundary and the
interior of $\D$. The former are based on arguments in \cite{P90} and
the latter constitute a generalisation of the non-convergence
criterion for the self vertex-reinforced random walk in \cite{P92}. We
illustrate our results by several examples while focusing on the
`competitive' version of the dynamics, namely when $\rpar_v^{ij} < 0$,
on complete graphs, star graphs and cycles. 

Interacting vertex and edge-reinforced random walks have been
considered in the machine learning community for graph clustering and
are known under the name of interacting particles, see \cite{SZ16} and
\cite{VUZ2018}. A different application of our model includes the
establishment of foraging patterns among competing species where food
resources are located at regions identified with the vertices of a
graph.  It is worth noting that for each vertex $v\in V$, the
numerator of $\pi_v^i(x)$ in \eqref{eqn:pi}, when $\alpha = 1$,
resembles the map $f^i:\R^m\to \R$ associated with the general
Lotka-Volterra equation for species $i$, namely $\dot x^i =
x^if^i(x)$, where $f^i(x) = r^i + \sum_{j=1}^m a_{ij} x^j$, and $r^i$
is the intrinsic growth rate of species $i$ and $A = (a_{ij})$ is the
$m\times m$ interaction matrix of the species at vertex $v$. From this
perspective, our model can be seen as several Lotka-Volterra models,
one per vertex, coupled by the structure imposed by $\G$. Few
deterministic versions including several models related to the one
described here have recently been considered in the literature, see
\cite{Slav2020}, \cite{CSSW2022} and \cite{LLC2024}. Taking a step
further, we observe that there exists a diffeomorphism mapping the
orbits of $\dot x^i = x^if^i(x)$ on $\R^m$ onto those of the
replicator equation on the unitary $m$-simplex, see \cite[Theorem
7.5.1]{HS98}. This conjugacy was explored by \cite{BT11}, enabling the
use of tools from evolutionary game theory to study the self
vertex-reinforced random walk in a large class of graphs. This
approach may also prove to be useful in our setting but we do not
pursue this here any further.


The remainder of this article is organised as follows. Our main
results are presented in Section~\ref{sec:main_results}.
Section~\ref{sec:examples} describes some examples to illustrate
asymptotic behaviours under various subgraph geometries and
reinforcement parameters. Proofs to all the results are provided in
Sections~\ref{sec:proof_Thdet} through \ref{sec:cycle_proofs}
and the Appendix.

\section{Statement of main results}\label{sec:main_results}
Let $\Fix(\pi) = \big\{x \in \D \mid x = \pi(x)\big\}$ be the set of
fixed points of $\pi$ defined by \eqref{eqn:THE_pi}.  For any given
subset of vertices $U \subset V$, denote by $\powerset(U)$ the power
set of $U$. Let
\[
  S = (S^1, \ldots, S^m) \in \V 
  \quad\text{where}\quad
  \V=\powerset(V^1) \times \cdots\times \powerset(V^m).
\]
Fix $S$ and then consider the following system of $d=\sum_{i\in [m]}
d^i$ linear equations in the real valued variables $x_v^i$, $i \in
[m]$ and $v \in V^i$, be defined for all $i\in [m]$ as
\begin{equation}\label{eqn:THE_system}
\left\{
\begin{array}{cl}
 \displaystyle
 \eta_v^i + \sum_{j \in I_v} \rpar_v^{ij} x_v^j = 
 \eta_w^i + \sum_{j \in I_w} \rpar_w^{ij} x_w^j, 
 \quad&  \forall v,  w \in S^i\ 
  \\[1.5em]
 \displaystyle
 \sum_{\ell \in S^i} x_\ell^i = 1, \quad \text{and} \quad  x_u^i = 0, 
 \quad& \forall u \in V^i{\setminus}S^i
\end{array}
\right.
\end{equation}
Let $\D(S)$ be the subset of solutions of \eqref{eqn:THE_system} such
that, for $x \in \D(S)$, $i \in [m]$ and $v \in S^i$, it follows that
$x_v^i > 0$.

Denote by $\limset$ be the random limit set of the vertex occupation
process $X = \{X(n); n\geq 0\}$. That is, $\limset$ stands for the set
of all $\omega$-limits, i.e. the accumulation points, of the process
$X$,
\begin{equation*}\label{eqn:Limit_Set_def}
  \limset = \bigcap_{n \geq 0} \overline{\bigcup_{k\geq n} \{X(k)\}}.
\end{equation*}

The first result of this article can now be stated as follows.

\begin{theorem}\label{th:det} 
Let $X$ be defined by \eqref{eqn:X} and $\pi$ by
\eqref{eqn:THE_pi}. The following statements hold:
\begin{enumerate}[$(i)$., nosep]
\item $\Fix(\pi)$ is the following finite union of connected sets
\begin{equation}\label{eqn:Fix}
    \Fix(\pi) = \displaystyle\bigcup_{S \in \V} \D(S).
\end{equation}
\item $\Fix(\pi) \neq \emptyset$. 
\item $\limset$ is almost surely connected.
\item 
\[
   \Prb\big(\limset \subset \Fix(\pi)\big) = 1.
\]
\item Let $R_\rpar(S)$ be the matrix of coefficients of the linear
system \eqref{eqn:THE_system} be defined by the interaction parameters
$\rpar_v^{ij}$. If $\det R_\rpar(S) \neq 0$ for all $S \in \V$, then
$\Fix(\pi)$ is finite and 
\[
  \sum_{x\, \in\, \Fix(\pi)} \Prb\Big(\lim_{n\to\infty} X(n) = x\Big)
  =  1. 
\]
\end{enumerate}
\end{theorem}

Our next result is perhaps the more general possible. It actually
shows that the convergence to some point of $\Fix(\pi)$ is the generic 
behaviour.

\begin{corollary}\label{cor:conv_is_generic}
Let $X$ be defined by \eqref{eqn:X} and $\pi$ by
\eqref{eqn:THE_pi}. For almost all choices of $\eta_v^i$, $\alpha$ and
$\rpar_v^{ij}$ subjected to \eqref{eqn:rho} and \eqref{eqn:eta}, $X$
converges almost surely to some point in $\Fix(\pi)$.
\end{corollary}

Although for almost all parameters, the set $\Fix(\pi)$ is finite,
there are examples in which some of the connected components of
$\Fix(\pi)$ are continua, i.e., compact connected subsets of $\D$.
Corollary~\ref{cor:emptyintesection}.\eqref{enu:C1} bellow provides in
particular one such example.  The following notation will be used
in Corollary~\ref{cor:emptyintesection} and throughout rest of
this article. For any $x \in \D$ with coordinates $x = (x^1, \ldots,
x^m)$, let $\Supp{x} = \Supp{x^1} \times \cdots \times \Supp{x^m}$
where
\[
 \Supp{x^i} = \big\{ v \in V^i : x^i_v > 0\big\}.
\]
The functions $\Supp{x}$ and $\Supp{x^i}$ denote respectively the
supports of $x$ and $x^i$.

\begin{corollary}\label{cor:emptyintesection}
Let $\pi$ be given as in \eqref{eqn:THE_pi}, and $x \in \D$ be such
that $ \Supp{x^i} \cap \Supp{x^j} = \varnothing$ for all $i, j \in
[m]$ with $i \neq j$ and $\eta_v^i = \eta_w^i$ for all $v, w
\in\Supp{x^i}$ and $i \in [m]$. Then, $x \in \Fix(\pi)$ if, and only
if, for every $i\in [m]$ one of the following conditions hold:
\begin{enumerate}[label=$(C_{\arabic*})$, ref=$C_{\arabic*}$, nosep,
  leftmargin=1cm]   
\item 
$\Supp{x^i} \subset \big\{ v \in V^i :
\rpar_v^{ii} = 0\big\}$,\label{enu:C1}
\item $\Supp{x^i} \subset \big\{ v \in V^i :
\rpar_v^{ii} > 0\big\}$ \, or \, $\Supp{x^i} \subset \big\{ v \in V^i :
\rpar_v^{ii} < 0\big\}$, \, and
\begin{equation}\label{eqn:explicit}
  x_v^i = \frac{1}{\rpar_v^{ii}}\left(\sum_{w \in \Supp{x^i}}
    \frac{1}{\rpar_w^{ii}}\right)^{-1}\quad \text{for all}\ \ v \in
  \Supp{x^i}.
\end{equation} 
\label{enu:C2}
\end{enumerate}
\end{corollary}

The condition $\eta_v^i = \eta_w^i$ in
Corollary~\ref{cor:emptyintesection} can be removed by adapting
\eqref{eqn:explicit}. The new expression for $x_v^i$ in
\eqref{eqn:explicit} 
is however more involved.

An explicit example of a continuum in $\Fix(\pi)$ is obtained by
considering two walks $W^1$ and $W^2$ on the complete three vertex
graph $G$, where $G^1 = G^2 = G$ and $\eta_v^i = \rpar_v^{ii} = 0$ for
all $i \in \{1,2\}$ and $v \in V=\{1,2,3\}$.  Observe that every $x
\in \D$ for which $\Supp{x^1} \subset \{1,2\}$ and $\Supp{x^2} =
\{3\}$ satisfies the condition \eqref{enu:C1} of
Corollary~\ref{cor:emptyintesection}. Since \eqref{enu:C1} does not
impose any restriction on the specific values of $x_v^i$, it is simple
to verify that the set $\{(a, 1-a, 0, 0, 0, 1) : a \in [0,1]\} \subset
\D$ is a continuum component of $\Fix(\pi)$. This situation is
depicted in Figure~\ref{fig:triangle}, where $W^1$ occupies the two
blue vertices at proportions $a$ and $1-a$ and $W^2$ is confined at
the green vertex.

\begin{figure}[h!]
  \begin{center}
  \begin{tikzpicture}[line width=0.65pt]
    \graph[nodes={draw, circle, inner sep=1pt, minimum
      size=2.5mm, line width=0.65pt}, 
      empty nodes, n=3, clockwise, radius=1.75cm] { 
      1[fill=blue!50]; 2[fill=green!50]; 3[fill=blue!50];
      subgraph K_n
    };
    \node[] at (-1.9, -1) {$a$};
    \node[] at (0.66, 1.9) {$1-a$};
  \end{tikzpicture} 
   \caption{A continuum component in $\Fix(\pi)$ for two competing
    walks on the tree vertex complete graph.} 
  \label{fig:triangle}
  \end{center}
\end{figure}
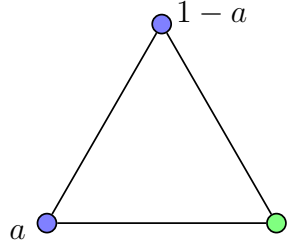

In general, the convergence of $X$ to a single point of a continuum
component of $\Fix(\pi)$ is an open question.
%
%

Theorems~\ref{th:no-toface} and \ref{th:no-tointerior} bellow provide
verifiable conditions to discard the convergence of $X$ to certain
points in $\D$. Both of these are instrumental in the
analysis of the examples presented in subsequent sections. Let
$\D^\circ$ be the interior of $\D$, that is $\D^\circ = \big\{x \in
\D\ \big|\ x_v^i > 0,\ v \in V^i, i \in [m]\big\}$ and let $\partial
\D=\D{\setminus}\D^\circ$ denote the boundary of $\D$.

\begin{theorem}[Non-convergence to boundary
  points]\label{th:no-toface}
Let $X$ be defined by \eqref{eqn:X} and $\pi$ by
\eqref{eqn:THE_pi}. Let $p \in \partial\D$ be such that
\begin{equation} \label{eqn:abar}
   p_k^i = 0 \quad \text{and} \quad \lim_{x \to p} \pi_k^i(x)/x_k^i > 
   1 \quad\text{for some}\quad i \in [m] \quad \text{and}
   \quad k \in V^i.
\end{equation}
Then
\[
   \Prb\Big(\lim_{n \to \infty} X(n) = p\Big) = 0. 
\]
\end{theorem}

Let $D\pi(x)$ be the Jacobian matrix of $\pi$ at $x$ and let
$\sigma\big(D\pi(x)\big)$ denote its spectrum. For any $\lambda \in
\mathbb{C}$, write $\mathfrak{R}(\lambda)$ for the real part of
$\lambda$.

\begin{theorem}[Non-convergence to interior points]
\label{th:no-tointerior} 
Let $X$ be defined by \eqref{eqn:X} and $\pi: \D\to \D$ by
\eqref{eqn:THE_pi}. Assume that $p \in \Fix(\pi)$, $p \in \iD$. If
$\mathfrak{R}(\lambda) \neq 1$ for all $\lambda \in
\sigma\big(D\pi(p)\big)$ and $\mathfrak{R}(\lambda) > 1$ for at least
one $\lambda \in \sigma\big(D\pi(p)\big)$, then
\[
   \Prb\Big(\lim_{n\to\infty} X(n) = p\Big) = 0.
\]
\end{theorem}

\section{Examples}\label{sec:examples}
This section presents several examples which illustrate previous
results in some concrete cases. The examples described here constitute
a small sample of the possibilities, far from being exhaustive in any
sense. They include interacting walks on a finite complete graph and
on complete subgraphs of star graphs and cycles. These were chosen
because of tractability and because we found them to be
interesting. We focus on a relatively simple version of the
competitive dynamics of our model where 
\begin{equation}\label{eqn:rhoij}
  \rpar_v^{ij} = -1 \quad \text{and} \quad \rpar_v^{ii} = -\epsilon,
  \quad  \text{where} \quad \epsilon \geq 0, 
\end{equation}
$v \in V$, $i, j \in I_v$ and $i \neq j$, recalling that  $I_v$ is
defined in \eqref{eqn:Iv}. 

We investigate the effect of $\epsilon>0$ versus $\epsilon=0$. As we
shall see, the limit properties of $X$ depend critically on the value
of $\epsilon$. In all but one example, we also assume that all the
vertices are `equally important' to each walk, namely
\begin{equation}\label{eqn:etaiv}
 \eta_v^i = \eta_w^j \quad \text{for all} \quad v,w \in V, \quad i  
 \in I_v \,\,   \text{ and } \,\,  j \in I_w.
\end{equation}

The following notation and definitions will be used to describe the
results throughout. Let $\ducal H:\D \to \R$ be defined by
\begin{equation}
  \label{eqn:H}
  \ducal H(x) = \sum_{i<j}\sum_v x_v^i x_v^j.
\end{equation}
The random variable $\ducal H\big(X(n)\big)$ corresponds to the
overlap measure of the vertex occupation process $X$ defined by the
random walks $W^i$, $i\in [m]$, up to time $n$. 

\subsection{Complete graphs}
Let $W^1$ and $W^2$ be two random walks sharing all of $K_\kappa$, the
complete graph with $\kappa \geq 2$ vertices. This is the case in
which both the subgraphs of $W^1$ and $W^2$ are equal to $K_\kappa$,
i.e.  $G^1=G^2=K_\kappa$ with $d^1 = d^2 = \kappa$.  The restriction
to $m = 2$ random walks is considered to simplify the exposition
throughout. This can be relaxed to $m > 2$ walks as explained in
Remark \ref{general} in the end of this section.

The results described in this section are similar in spirit to those
in \cite{C14}, who considered two repelling walks on complete graphs
evolving according to the transition probability
\[
   \Prb\big(W^i(n+1) = v \mid \F_n\big)
   \propto
   \frac{1}{\max\! \big\{X_v^j(n), \delta\big\}^\alpha},
\]
where $\alpha>0$, $\delta>0$. Similarly as in the case of competitive
dynamics considered here, see \eqref{eqn:trans_prob} and
\eqref{eqn:pi} with $\rpar_v^{ij} < 0$, the probability of a
transition to vertex $v$ by one walk decreases with the proportion of
visits to that vertex by the other walk. The principal result in
\cite{C14} concerns the limiting behaviour of the overlap measure of
the joint vertex occupation process, $\ducal{H}(X(n))$. For
sufficiently large $\alpha$, \cite{C14} shows that $\limsup_n
\ducal{H}(X(n)) < \delta$ holds almost surely. By virtue of
Theorem~\ref{th:det}, we show here that $X$ does indeed converges for
almost all values of $\eta_v^i$, $\rpar_v^{ij}$, $\alpha$, and as a
consequence $\ducal{H}(X(n))$ also does.  Further, here we are able to
identify the points to which $X$ converges to. As an example,
Theorem~\ref{th:complete_G_epsilon} bellow gives an explicit
description of the limit points of $X$ assuming \eqref{eqn:rhoij} and
\eqref{eqn:etaiv} for $\epsilon > 0$.  Theorem~\ref{th:complete_G}
also gives a complete description of the limit set according to which
the walks coordinate themselves to visit different vertices when
$\epsilon = 0$. As a side result, it holds almost surely that $\lim_n
\ducal{H}(n)$ always exists. More precisely, in the latter case, when
$\epsilon = 0$, it follows that $\lim_n \ducal{H}(n) = 0$ whereas in
the former, when $\epsilon > 0$, we have that $\lim_n \ducal{H}(n) <
\epsilon$ holds for sufficiently small $\epsilon$.

Before stating the main results of this section,
Figure~\ref{fig:complete} displays two representative elements of the
limit set of the process $X$ on the complete graphs $K_5$ and
$K_9$. White vertices are left unoccupied, green vertices are occupied
by one walk and the ones in blue by the other walk.
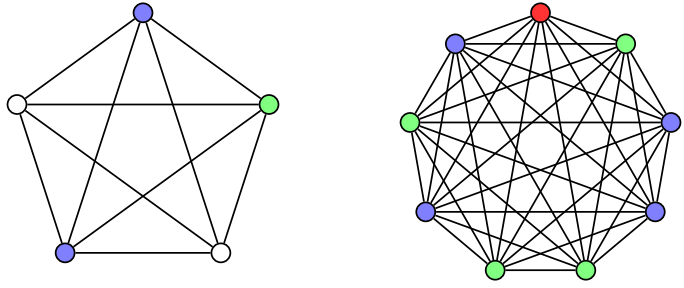
\begin{figure}[h!]
{\centering
\begin{tikzpicture}[line width=0.65pt]
  \begin{scope}[shift={(0,0)}]
    \graph[nodes={draw, circle, inner sep=1pt, minimum
      size=2.5mm, line width=0.65pt}, 
      empty nodes, n=5, clockwise, radius=1.75cm] { 
      1[fill=blue!50]; 2[fill=green!50]; 3[draw];
      4[fill=blue!50]; 5[draw]; 
      subgraph K_n
    };
  \end{scope}
  \begin{scope}[shift={(5.25,0)}]
    \graph[nodes={draw, circle, inner sep=1pt, minimum
      size=2.5mm, line width=0.65pt}, 
      empty nodes, n=9, clockwise, radius=1.75cm] {
      1[fill=red!80];
      2[fill=green!50]; 3[fill=blue!50];
      4[fill=blue!50]; 5[fill=green!50]; 6[fill=green!50];
      7[fill=blue!50]; 8[fill=green!50]; 9[fill=blue!50];
      subgraph K_n
    };
  \end{scope}
\end{tikzpicture}
\par}
\caption{Possible limits for the vertex occupation by two
  walks on $K_5$ (left) and $K_9$.
}\label{fig:complete}
\end{figure}
More precisely, a possible limit on $K_5$, assuming \eqref{eqn:rhoij}
and \eqref{eqn:etaiv} for $\epsilon = 0$, is determined by setting the
occupation of the two blue vertices respectively to $a$ and $1-a$
where $a \in (0,1)$. The occupation depicted in $K_9$ is an instance
of the limits obtained for $\epsilon = \frac{4}{20}$. Blue and green
vertices are each occupied at a proportion equal to $\frac{6}{25}$
while the red vertex, shared by both walks, is occupied at a
proportion equal to $\frac{1}{25}$.

\begin{theorem}[Complete graphs, $\epsilon = 0$]
\label{th:complete_G} 
Let $K_\kappa$ be the complete graph with $\kappa\geq 2$ vertices and
let $W^1(n)$, $W^2(n)$ be two random walks on $K_\kappa$, with
transition probabilities defined according to
\eqref{eqn:THE_pi}. Suppose \eqref{eqn:rhoij} and \eqref{eqn:etaiv}
hold. If $\epsilon = 0$, then the vertex occupation measure
$X$, defined by \eqref{eqn:X}, converges almost surely towards the set
\begin{equation}\label{eqn:K}
 \eK = \Big\{
 (p^1, p^2) \in \D\ \Big|\ \Supp{p^1} \cap \Supp{p^2} = 
 \varnothing \Big\},
\end{equation}
that is, $\Prb(\limset \subset \eK) = 1$. 
\end{theorem}

\begin{corollary}\label{cor:overlap_complete_G_eps_zero}
Let $K_\kappa$ be the complete graph with $\kappa\geq 2$ vertices and
let $W^1(n)$, $W^2(n)$ be two random walks on $K_\kappa$, with
transition probabilities defined according to
\eqref{eqn:THE_pi}. Suppose \eqref{eqn:rhoij} and \eqref{eqn:etaiv}
hold. If $\epsilon = 0$, then
\[
  \Prb\Big(\lim_{n\to\infty} \ducal{H}(X(n)) = 0 \Big) = 1.
\]
\end{corollary}

Although we can not say that $X$ converges almost surely to an
isolated point $x \in \D$ when $\epsilon = 0$, $X$ does converges to a
point with arbitrarily small vertex overlap measure if we allow
$\epsilon$ to be slightly positive. The next theorem identifies the
asymptotic form of the vertex occupation measure process $X$
for sufficiently small $\epsilon > 0$. For any $p=(p^1, p^2) \in \D$
and $i =1, 2$, let
\begin{equation}\label{eqn:sizes}
  \begin{gathered}
    U(p) = \Supp{p^1} \cap \Supp{p^2}, \qquad
    U^i(p) = \Supp{p^i}{\setminus}U(p)\\
    u(p) = |U(p)|,  \qquad s^i(p) = |\Supp{p^i}|, \quad\text{and}\quad 
    u^i(p) = s^i(p) - u(p).
  \end{gathered}
\end{equation}
To simplify notation, we will hereafter drop the dependence on $p$ of
the quantities in \eqref{eqn:sizes}.

\begin{theorem}[Complete graphs, $\epsilon > 0$]
\label{th:complete_G_epsilon}
Let $K_\kappa$ be the complete graph with $\kappa\geq 2$ vertices and
let $W^1(n)$, $W^2(n)$ be two random walks on $K_\kappa$, with
transition probabilities defined according to
\eqref{eqn:THE_pi}. Suppose \eqref{eqn:rhoij} and \eqref{eqn:etaiv}
hold.  Let $\eK$ be defined as in \eqref{eqn:K}, $\eK^c =
\D{\setminus}\eK$, and $\eK_1^c = \big \{ \, p \in \eK^c | \ u^1 > 0
\, \text{ and } \, u^2 > 0 \big \}$.  If $\epsilon > 0$, then, for
sufficiently small $\epsilon$, the vertex occupation measure $X$
converges almost surely towards a single point $p \in \D$ of one of
the following sets
\begin{gather} 
  \label{eqn:unifabitra}
  \tilde{\eK} =  \bigg \{ \, p \in \eK
     \ \bigg| \
       p_v^i = \frac{1}{s^i} \,\, \text{ for } \,\, v \in \Supp{p^i} \,
       \text{ and } \, i=1,2 \bigg \} \\
  \label{eqn:overlappted}
  \tilde{\eK}_1^c =
    \left\{ \, p \in \eK_1^c
     \ \bigg| \
        p_v^i =
        \begin{cases}
          \delta^i,& v \in U\\
          \bar{\delta}^i,&  v \in U^i
        \end{cases},
        \text{ for } \, i=1,2
     \right\},
\end{gather}
where for $j = 3 - i$,
\begin{equation}\label{eqn:deltaEbardelta}
  \delta^i = \frac{\epsilon(\epsilon s^j -
  u^i)}{\epsilon^2 s^i s^j - u^i u^j} 
  \quad \text{and} \quad 
  \bar \delta^i = \frac{1}{u^i} - \frac{u}{u^i} 
  \delta^i. 
\end{equation}
\end{theorem}

\begin{corollary}\label{cor:overlap_complete_G_eps_pos} Let $K_\kappa$
be the complete graph with $\kappa\geq 2$ vertices and let $W^1(n)$,
$W^2(n)$ be two random walks on $K_\kappa$, with transition
probabilities defined according to \eqref{eqn:THE_pi}. Suppose
\eqref{eqn:rhoij} and \eqref{eqn:etaiv} hold. If $\epsilon > 0$, then
the following two assertions hold:
\begin{enumerate}[$(i)$., nosep]
\item the limit $\lim_{n\to\infty} \ducal{H}(X(n))$
  exists almost surely,
\item $\Prb\big(\lim_n \ducal{H}(X(n)) < \kappa^3\epsilon^2\big) = 1$ for 
  sufficiently small $\epsilon > 0$.
\end{enumerate}
\end{corollary}

\begin{remark}\label{general}
As a generalisation to Theorem~\ref{th:complete_G_epsilon}, one could
consider $m\geq 2$ interacting walks on $K_\kappa$. For sufficiently
small but still positive $\epsilon$, a simple computation shows that
$\det R_\rpar(S) \neq 0$ for any $S \in \V$. As a consequence, from
Theorem~\ref{th:det}.$(iv)$ it follows that the vertex occupation
process $X$ converges almost surely to a single point in
$\Fix(\pi)$. A description of the possible limits for $X$ and $\ducal
H(X(n))$ are interesting problems but we do not pursue this here any
further.
\end{remark}

\subsection{Star graphs}
Throughout, let $G$ be a star graph with $m$ edges, i.e. a complete
bipartite graph $K_{1,m}$. Let
\begin{equation}\label{eqn:stars_subgraphs}
\G = \big\{G^i = (V^i, E^i);\ i \in [m]\big\}, \ \text{ where }\ 
V^i = \{1, i+1\},\  E^i = \{\langle 1, i+1 \rangle\}
\end{equation}
be the set of the complete subgraphs of $G$, where the central vertex
of the star is labelled as 1.  Observe that $\G$ satisfying
\eqref{eqn:stars_subgraphs} is completely parameterised by $m$.

Suppose that \eqref{eqn:rhoij} and \eqref{eqn:etaiv} hold. When
$\epsilon = 0$, a single walk occupies the central vertex at a
proportion $a \in [0,1]$; as $\epsilon$ increases, this behaviour
gives way to a regime in which $1 \leq k \leq m$ walks share the
central vertex. A precise description of the limiting vertex
occupations by each random walk is given by the following theorem.

\begin{theorem}[Star graphs]\label{th:star_G}
Let $G$ be the star graph with $m$ subgraphs defined by
\eqref{eqn:stars_subgraphs}.  Let $X$ be the vertex occupation process
in~\eqref{eqn:X}. Suppose that \eqref{eqn:rhoij} and \eqref{eqn:etaiv}
hold.

If $\epsilon = 0$, then $X$ converges almost surely to
\[
  \eS_1
  =
  \Big\{
  p \in \D \ \Big|\ p_1^i \in [0, 1]\ \text{ for some } i \in [m]\
  \text{ and }\ p_1^j = 0\ \text{ for }\ j \neq i
  \Big\},
\]
that is, $\Prb\big(\limset \subset \eS_1 \big) = 1$.

If $\epsilon \in \big(0, \frac12\big)$, then $X$ converges almost
surely to one of the points of the set $\eS_2$ defined as
\[
 \eS_2
  =
  \Big\{p \in \D\ \Big|\ 
  p_1^i = \frac{\epsilon}{|K| + 2\epsilon-1}\ \text{ for } i
  \in  K,\  \text{ and } \ p_1^j=0\  \text{ for } j \in
  [m]{\setminus}K
  \ \text{ and } \varnothing \neq K \subsetneq [m]
  \Big\}.
\]

If $\epsilon \in (\frac12, 1)$, then $X$ converges almost surely 
towards the point
\[
  \Big(
  \frac{\epsilon}{m + 2\epsilon-1}, 1- \frac{\epsilon}{m +
    2\epsilon-1}, \ldots, \frac{\epsilon}{m + 2\epsilon-1},
  1- \frac{\epsilon}{m + 2\epsilon-1}  
  \Big).
\]
\end{theorem}

Theorem~\ref{th:star_G} admits the following corollary, which asserts
that, for sufficiently small $\epsilon$, the occupation measure of the 
central  vertex can be made arbitrarily small in the long run. 

\begin{corollary}\label{cor:star_overlap} Under the conditions of
  Theorem~\ref{th:star_G}, if $\epsilon \in \Big(0, 
\frac{2}{m(m-1)}\Big) \mathbin{\big\backslash} \big\{\frac12\big\}$
then
\[
  \Prb\Big(\lim_{n\to\infty}\ducal{H}(X(n)) \leq \epsilon\Big) 
  =  1.
\]
In addition, when $\epsilon = 0$, then $\lim_n \ducal{H}(X(n)) =
0$ holds almost surely.
\end{corollary}

The next result provides an example of two random walks on the three
vertex star graph in the competitive case but with distinct vertex
importance parameters $\eta_v^i$. Transitions to the common central
vertex are favoured over those to the other vertices at the extremes
of the graph, by an importance weight of $\tilde\eta+\eta$ against
$\tilde\eta$.  The limit behaviour of $X$ has in this case a phase
transition determined by $\eta$, the additional importance of the
central vertex. 

\begin{lemma}[Star with preferences]\label{lem:star_with_preferences}
Let $G$ be the star graph with $m=2$ and subgraphs defined by
\eqref{eqn:stars_subgraphs}.  Let $X$ be defined by
\eqref{eqn:X}. Suppose that \eqref{eqn:rhoij} holds with $\epsilon =
0$. For $\tilde\eta > 0$, set $\eta_v^i = \tilde\eta + \eta$ if
$v=1$ and $\eta_v^i = \tilde\eta$ otherwise, where $\tilde \eta$ is
sufficiently large such that \eqref{eqn:eta} is satisfied for all
$\eta > 0$. If $0 < \eta < 1$, $X$ converges almost surely to one of
the points $(1, 0, 0, 1)$ or $(0, 1, 1, 0)$. If $\eta > 1$, $X$
converges almost surely to $(1, 0, 1, 0)$.
\end{lemma}

\subsection{Cyclic graphs}
This section describes the asymptotic properties of $X$ when $G$
corresponds to a cyclic graph with $m \geq 3$ edges, $C_m$.
The set of complete subgraphs is in this case given by
\begin{equation}\label{eqn:cicle_subgraphs}
  \G = \big\{G^i = (V^i, E^i);\ i \in [m]\big\},
  \quad V^i = \{i, i+1\},
  \quad E^i = \{\langle i, i+1 \rangle\},
\end{equation}
for $i=1, \ldots, m-1$, and 
$V^m =\{m,1\}$, $E^m = \{\langle m, 1\rangle\}$.

Again, we
focus on the competitive version of our model by taking $\rpar^{ii}_v
= 0$ and $\rpar_v^{ij} = -1$.  As one would expect, $X$ can converge
to a point $p \in \D$ with `total segregation', that is, a point $p$
such that $\ducal{H}(p)=0$. Yet, the competitive model just described
is relatively rich and presents quite surprisingly several limiting
behaviours.

We will make use of the following notation in order to state the main
results of this section. Let $v \in V$, and then for $\langle v,
v+1\rangle \in E$ and $a \in [0,1]$ denote by $\dbr{a, 1-a}$ the
successive values of the coordinates $p_v^v$ and $p_{v+1}^v$ of a
point $p \in \D$. By a minor abuse of language, we refer hereafter to
$\dbr{a, 1-a}$ as an edge.  An edge is said to be unmixed if $a = 0$
or $a=1$, and mixed otherwise. An unmixed edge is oriented in the
sense that it can be either of the form $\dbr{1, 0}$ or $\dbr{0,
1}$. We eventually may write $\triangle^m$ instead of $\D$.  Finally,
let $\eC_i$, $1 \leq i \leq 4$, be subsets of $\triangle^m$ defined as
follows. Let
\[
  \eC_1 = \Big\{p \in \triangle^m\ \Big|\ p =
  \big(\!\dbr{\textstyle\frac12, \frac12}, \ldots,
  \dbr{\textstyle\frac12,\frac12}\!\big) 
  \Big\},
\]
If $m$  is not a multiple of four, then $\eC_2 = \varnothing$,
else 
\[
  \eC_2 = \bigg\{ p \in \triangle^m\ \Big|\ p =
  \big(\underbrace{p_{ab}, \ldots, p_{ab}}_k\big), \  a, b  \in
  (0,1)\bigg\}
\]
where $p_{ab} = \big(\!\!\dbr{a, 1-a}$, $\dbr{b, 1-b}$, $\dbr{1-a,
a}$, $\dbr{1-b, b}\!\!\big)$ and $k\in \N$ is such that $k =
m/4$. Finally,  let
 \[
  \eC_3 = \Big\{p\in \triangle^m \mid p^i_v \in
  \{0,1\}, i \in [m], v = i\Big\},
\]
and  
\[
  \eC_4 = \Big\{
  p\in \triangle^m\ \Big|\
  \text{mixed edges are flanked by unmixed edges of opposite
    direction} 
  \Big\}.
\]

\begin{theorem}[Cyclic graphs, $\epsilon =0$]\label{th:cyclic_G}
Let $G$ be the cyclic graph with $m$ subgraphs defined by
\eqref{eqn:cicle_subgraphs}.  Let $X$ be the vertex occupation process
defined by \eqref{eqn:X}. Suppose that \eqref{eqn:rhoij} and
\eqref{eqn:etaiv} hold. If $\epsilon = 0$, then 
\begin{enumerate}[($i$),nosep]
\item $\Prb\Big(\limset \subset \textstyle\bigcup_{r=1}^4 \eC_r =
  \Fix(\pi)\Big) = 1$, and
\item If $m$ is not a multiple of four, then $\Prb\Big(\limset \subset
  \eC_3 \cup \eC_4 \subset \Fix(\pi)\Big) = 1$,
\end{enumerate}
where $\pi$ is defined by \eqref{eqn:THE_pi}.
\end{theorem}

\begin{corollary}\label{cor:if_X_converges}
If the joint process $X$ converges almost surely to a single point,
then this limit point must be an element of $\tilde\eC_1 \cup
\tilde{\eC}_3 \cup \tilde{\eC}_4$. The sets $\tilde\eC_1$,
$\tilde{\eC}_3$ and $\tilde{\eC}_4$ are defined as
\begin{gather}
  \tilde\eC_1 = 
  \begin{cases}
     \eC_1,& \text{ if $m$ is a multiple of four},\\
     \varnothing,& \text{ else}.
  \end{cases}
  \notag \\
  \label{eqn:C3tilde}
  \tilde{\eC}_3 =
  \bigg\{
   p \in \eC_3\ \bigg|\ 
   \text{%
      \parbox{7cm}{%
      up to cycle orientation, $p$ does not presents the sequence
      $\ldots$, $\dbr{1,0}$, $\dbr{0,1}$, $\dbr{1,0}$, $\ldots$ 
      }
   }
  \bigg\},
\end{gather}
and 
\begin{equation}
 \label{eqn:C4tilde}
  \tilde{\eC}_4 =
  \left\{
    p \in \eC_4\ \left|\
    \text{%
       \parbox{9cm}{%
        up to cycle orientation, $p$ does not presents the sequence
        $\ldots$, $\dbr{1-a, a}$, $\dbr{0,1}$, $\dbr{1,0}$, $\ldots$  
        nor the sequence $\ldots$, $\dbr{1, 0}$, $\dbr{0,1}$,
        $\dbr{b,1-b}$, $\ldots$ for $a \in [0,1)$ and $b \in (0,1]$
      }
    } 
    \right.
  \right\}.
\end{equation}
\end{corollary}  

\begin{theorem}[Cyclic graphs, $\epsilon > 0$]\label{th:cyclic_G_e} 
For each $i\in [m]$, let $W^i$ be a random walk defined on the $i$-th
edge of $C_m$ according to the transition probability in
\eqref{eqn:pi}.  Suppose that \eqref{eqn:rhoij} and \eqref{eqn:etaiv}
hold. For sufficiently small $\epsilon>0$, the vertex occupation
process $X$ converges almost surely to a single point of \
$\Fix(\pi)$.
\end{theorem}

\begin{remark}
In the competitive dynamics considered by Theorem~\ref{th:cyclic_G_e},
the set $\Fix(\pi)$ may have continuum components depending on $S$,
$m$, and $\epsilon$. For given $S\in \V$ and $m\geq 3$, the values for
$\epsilon$ where this occurs belong to the set of real algebraic
numbers of $(0, 1]$.  For instance, when $S\in\V$ is such that $S^i =
V^i$ for all $i\in [m]$, there is a continuum of fixed points of $\pi$
when $m = 3$ and $\epsilon=1$. When $m=16$, this occurs if $\epsilon
\in \{0, 1/\sqrt{2}, \cos(\pi/8), 1\}$, and when $m=24$ if
$\epsilon=\big\{0, 1/2, 1/\sqrt{2}, \sqrt{3}/2,
\frac14\big(\pm\sqrt{2}+\sqrt{6}\big), 1\big\}$. These values for
$\epsilon$ and those for any other $S\in \V$ and $m\geq 3$ can be
deduced from the arguments used in the proof of
Theorem~\ref{th:cyclic_G_e}.
\end{remark}

\section{Proofs of Theorem~\ref{th:det} and
Corollaries~\ref{cor:conv_is_generic},
\ref{cor:emptyintesection}}\label{sec:proof_Thdet}

\subsection{Proof of Theorem~\ref{th:det}}

\subsubsection{Proof of assertion $(i)$} 
Suppose first that $x \in \Fix(\pi)$. Let $S \in \V$ and $S^i =
\Supp{x^i}$ be such that $v, w \in S^i$.  We will show that $x \in
\D(S)$. Since $x \in \Fix(\pi)$ it follows that $x_v^i = \pi_v^i(x)$
and $x_w^i = \pi_w^i(x)$. Let $N^i(x)$ be the common normalisation
factor of $\pi_v^i(x)$ and $\pi_w^i(x)$, that is, the denominator of
the fraction in equation (\ref{eqn:pi}).  Using the fact that $x_v^i >
0$ and $x_w^i > 0$ and multiplying the two previous equations by
$N^i(x)/x_v^i$ and $N^i(x)/x_w^i$ respectively gives
\begin{equation}\label{eqn:normalising}
  H_v^i(x) = N^i(x) = H_w^i(x), \quad v, w \in S^i
\end{equation}
where $H_v^i(x)$ is defined in \eqref{eqn:pi}.

The first set of equations in \eqref{eqn:THE_system} follow by raising 
all the expressions of \eqref{eqn:normalising} to the power
$1/\alpha$. The two other sets of equations in \eqref{eqn:THE_system}
follow because by assumption $x = \pi(x) \in \D$, and therefore
$\sum_{\ell \in S^i} x_\ell^i =1$ and $x_u^i = 0$ when $u \in
V^i{\setminus}S^i$.

Assume now that $x \in \D(S)$ for some $S \in \V$. We will show
that $x \in \Fix(\pi)$.  If $x_v^i = 0$, then $\pi_v^i(x) = 0$ by
definition of $\pi$ in (\ref{eqn:pi}). If $x_v^i > 0$, then $v \in
S^i$ and it follows by \eqref{eqn:THE_system} and \eqref{eqn:pi} that
$H_v^i(x)$ is equal to $H_w^i(x)$ for all $w \in S^i$.  As a
consequence, the denominator of the fraction in (\ref{eqn:pi}) reduces
to
\[ 
 N^i(x) = \sum_{w \in V^i} x_w^i H_w^i(x)  = \sum_{w \in S^i} x_w^i
  H_w^i(x)  = \sum_{w \in S^i} x_w^i H_v^i(x)  = H_v^i(x)
\]
Substitution of the denominator $N^i(x)$ in (\ref{eqn:pi}) by
$H_v^i(x)$ now gives
\[
  \pi_v^i(x) = x_v^i\frac{H_v^i(x)}{N^i(x)} =
  x_v^i\frac{H_v^i(x)}{H^i_v(x)} = x_v^i.
\]
This establishes the equality in~\eqref{eqn:Fix}.  The set $\V$ has
cardinality $2^{\sum_{i=1}^m |V^i|}$ and each of the sets $V^i$ is
finite, the union in \eqref{eqn:Fix} is therefore finite. The rest of
the proof consists in showing that $\D(S)$ is connected. To see this,
it is sufficient to show that $\D(S)$ is the intersection of two
convex subsets of $\R^d$ where we recall that $d=\sum_{i\in [m]} d^i$.
The first subset is the set of solutions of $d$ linear equations of
$d$ real variables defined by \eqref{eqn:THE_system} being therefore a
convex subset of $\R^d$.  The second subset is the set $\{x \in \R^d |
x_v^i > 0, v \in S^i, i \in [m]\}$ which is obviously a convex subset
of $\R^d$ as well.

\subsubsection{Proof of assertion $(ii)$ of Theorem~\ref{th:det}}
In view of the first item in the theorem, it is sufficient to show
that there exists one $S \in \V$ such that for any $\eta_v^i$,
$\rpar_v^{ij}$ and $\alpha$, the corresponding set of solutions to
\eqref{eqn:THE_system}, $\D(S)$, is non-empty. To this end let $S$ be
such that $|S^i| = 1$ for all $i \in [m]$ and let $S^i= \{v^i\}$, $v^i
\in V^i$. The system \eqref{eqn:THE_system} admits in this case the
following trivial solution for any $\rpar_v^{ij}$, $\eta_v^i$ and
$\alpha$: $x_{v^i}^i = 1$ and $x_v^i = 0$ for all $v \in V^i$ such
that $v \neq v^i$. To see this, note that the second equation in
\eqref{eqn:THE_system} reduces to $x_\ell^i = 1$, while the first only
applies when $|S^i| \geq 2$. This concludes the proof.

\subsubsection{Proof of assertion $(iii)$ of Theorem~\ref{th:det}}
The proof of item $(iii)$ and a part of the proof of item $(iv)$ from
Theorem~\ref{th:det} is based on the dynamical systems treatment of
stochastic approximations as presented in \cite{B96}. A
first step consists in identifying the joint vertex occupation process
$X$ as a stochastic approximation, see Lemma~\ref{lem:SA}. This allows
to relate the long term behaviour of $X$ to the one of a flow defined
by an underlying vector field, see Lemma~\ref{lem:limitset}. Global
convergence results are then established by means of a Lyapunov
function. The construction of this function is obtained as a
generalisation of the arguments concerning a single self reinforced
random walk in \cite{B15} and can be traced back to \cite{BudhirajaI}.
This is possibly the main contribution of
Section~\ref{sec:proof_Thdet}. This generalisation may also be of
independent interest.

We start by introducing some preliminary definitions and
results. Recall that $d = \sum_{i \in [m]} d^i$, where $d^i = |V^i|$
and $V^i$ is the set of vertices of the subgraph $G^i$ of $G$. Denote
by $\TsD$ the tangent space of $\D$, that is
\[
  \TsD
  =
  \bigg\{\big(x_{\underline{1}}^1, \ldots, x_{{\underline{d}}^1}^1, 
  \ldots, x_{\underline{1}}^m, \ldots, x_{{\underline{d}}^m}^m\big)
  \in \R^d\ \bigg|\ \sum_{v=1}^{d^i} x_{\underline{v}}^i = 0,\ i \in
  [m]\bigg\}.
\]

\begin{lemma}[Stochastic approximation]\label{lem:SA} Let
$\G$ be a set of complete subgraphs of $G$ and $X=\{X(n);
n\geq 0\}$ be the vertex occupation measure process in \eqref{eqn:X}
be defined by the transition probability $\pi$ in \eqref{eqn:pi}. $X$
satisfies the recursion
\begin{equation}
  \label{eqn:SA}
     X(n+1) - X(n) = \big(F(X(n))+U(n)\big) \iG_n,
\end{equation}
where $F:\D \to \TsD$ is the smooth vector field 
\begin{equation}
  \label{eqn:THE_field}
  F(x) = -x  + \pi(x),
\end{equation}
$U(n)$ is the martingale difference $U(n) = \xi(n) - \E[\xi(n)\mid
\F_n]$, with $\xi(n) = (\xi^1(n), \ldots, \xi^m(n))$, $\xi^i(n) =
(\xi^i_{\underline{1}}(n), \ldots, \xi^i_{{\underline{d}^i}}(n))$,
$\xi^i_v(n) = \Ind\big\{W^i(n+1) = v\big\}$, and $\Xi_n$ is a
$d$-by-$d$ diagonal matrix formed by $m$ diagonal $d^i\times d^i$
block matrices $\Xi^i_n$, each with all diagonal entries equal to
$\gamma_n^i = 1/(d^i + n + 1)$, in the order $\Xi^1_n$, $\Xi^2_n$,
$\ldots$, $\Xi^m_n$.
\end{lemma}

The proof Lemma~\ref{lem:SA} is presented in the Appendix.

Observe that for any $\rpar_v^{ij} \in \R$, $\alpha > 0$ and
$\eta_v^i$ satisfying \eqref{eqn:eta}, the vector field $F$ in
\eqref{eqn:THE_field} is Lipschitz continuous as all derivatives
$\partial \pi_v^i(x)/\partial x_w^j$, $i, j \in [m]$, $v \in V^i$, $w
\in V^j$, $x\in\D$ are uniformly bounded. In this case, according to
standard results, there exists a uniquely defined one-parameter family
$\Phi=\{\phi_t : t \in [0, \infty)\}$ of self-maps of $\D$, called the
semi-flow induced by $F$ that satisfies $\phi_0(x_0) = x_0$ and
$\phi_t(x_0) \in \iD$ for all $\quad t \geq 0$, and
\begin{equation}\label{eqn:THE_flow} 
 \frac{d}{dt} \phi_t(x_0) = F(\phi_t(x_0)) \quad\text{for all}\quad t
 \geq 0.
\end{equation}

\begin{definition}\label{def:flow}
A set $\Lambda \subset\D$ is forward invariant by $\Phi$  if $\phi_t(x)
\in \Lambda$ for all $t\geq 0$ and $x \in \Lambda$. The orbit of $x
\in \D$ by $\Phi$ is the set $\{\phi_t(x) : t \geq 0\}$. Let $\delta
>0$, $T>0$. A $(\delta, T)$-pseudo orbit from $x \in \D$ to $y \in \D$
is a finite sequence of partial orbits $\{\phi_t(y_i) : 0 \leq t \leq
t_i\}$; $i=0, \ldots, k-1$ of the semi-flow $\Phi$ such that
\[
  \|y_0 - x\| < \delta, \quad \|\phi_{t_i}(y_i) - y_{i+1}\| <
  \delta \ \text{ for }\ 
  t_i \geq T, \quad i=0, \ldots, k-1, \ \text{ and }\  y_k = y. 
\]
A point $x \in \D$ is \emph{chain-recurrent} if for every
$\delta>0$ and $T>0$ there is a $(\delta, T)$-pseudo orbit from $x$ to
itself. The set of chain-recurrent points of $\Phi$ will be
denoted by $\CR$.
\end{definition}

Lemma~\ref{lem:limitset} bellow makes precise the relation between the
asymptotic behaviour of the process $X$ and the flow defined by the
smooth vector field $F$ in \eqref{eqn:THE_field}.

\begin{lemma}[Limit set lemma]\label{lem:limitset}
Let $X$ be the vertex occupation measure process in \eqref{eqn:X} be
defined by the transition probability $\pi$ in \eqref{eqn:pi}. Let
$\CR$ be the chain recurrent set of the semi-flow induced by the
vector field $F$ in \eqref{eqn:THE_field}. The following two
statements hold almost surely
\begin{enumerate}[$(i)$., nosep]
\item $\limset$ is connected.
\item $\limset \subset \CR$.
\end{enumerate}
\end{lemma}

The proof of Lemma~\ref{lem:limitset} relies on Lemma~\ref{lem:SA} and
is presented in the Appendix.  In particular, Assertion $(i)$ of
Lemma~\ref{lem:limitset} establishes Assertion $(iii)$ of Theorem
\ref{th:det}.
\subsubsection{Proof of assertion $(iv)$ of Theorem~\ref{th:det}}
A $C^0$ function $L:\D\to\R$ is said to be a strict Lyapunov function
for the forward invariant set $\Lambda \subset \D$ of the semi-flow
$\Phi$ induced by the vetor field $F$, or simply a strict Lyapunov
function for $\Lambda$, if for all $x \in \D$ the map $t \in \R_+
\mapsto L(\phi_t(x))$ is constant for $x \in \Lambda$ and strictly
decreasing for $x \in \D{\setminus}\Lambda$.  Observe that if
$\big\langle \nabla L(x), F(x)\big\rangle < 0$ holds for all $x \in
\D$ such that $F(x) \neq 0$ and $F$ is the vector field in
\eqref{eqn:THE_field}, then $L$ is a strict Lyapunov function for
$\Lambda = \Fix(\pi)$.

\begin{theorem}[Lyapunov function]\label{th:Lyapunov} 
The function $L:\D \to \R$ defined by
\begin{equation}\label{eqn:Lyapunov-pwr-H} 
  L(x)  = - \sum_{i\in [m]} \sum_{v\in V^i}  \eta_v^i  x_v^i - \frac12
  \sum_{v \in V}\sum_{i\in I_v}\sum_{j\in I_v} \rpar_v^{ij}x_v^i
  x_v^j.
\end{equation}
is a strict Lyapunov function for $\Fix(\pi)$.
\end{theorem}

\begin{proof}
Let $\varphi:(0,\infty)\to\R$ be defined by $z \mapsto -z^{-1/\alpha}$
and by some abuse of notation, for any $x \in \D$, $x^i =
(x^i_{\underline{1}}, \ldots, x^i_{\underline{d}^i})$ and $\pi^i(x) =
\big(\pi^i_{\underline{1}}(x), \ldots,
\pi^i_{\underline{d}^i}(x)\big)\in \triangle^i$, write 
\[
  \varphi(x^i/\pi^i(x))
  =
  \Big(\varphi\big(x^i_{\underline{1}}/\pi^i_{\underline{1}}(x)\big),
  \ldots, 
  \varphi\big(x^i_{\underline{d}^i}/\pi^i_{\underline{d}^i}(x)
  \big)\Big).    
\]
For $x \in \D$ and $i \in [m]$ define $\Pi^i(x)$ as the
$d^i$-by-$d^i$ matrix with all rows equal to $\pi^i(x)$. Now, let 
$\Gamma(x)$ be the $d$-by-$d$ rate matrix given by
\[
  \Gamma(x) = -I + \Pi(x),
\]
where $I$ is the $d$-by-$d$ identity matrix and $\Pi(x)$ is a
$d$-by-$d$ block matrix constituted by the matrices $\Pi^i(x)$
disposed along the diagonal in the order $\Pi^1(x)$, $\Pi^2(x)$,
$\ldots$, $\Pi^m(x)$, and zero block matrices elsewhere.  Let
$\Gamma^i(x) = -I^i + \Pi^i(x)$, where $I^i$ is the $d^i$-by-$d^i$
identity matrix.  We will also consider the reduced forms of $x^i$,
$\pi^i(x)$, $x^i/\pi^i(x)$ and $\Gamma^i(x)$. That is, for given $x
\in \D$, let $\tilde x^i$, $\tilde \pi^i(x)$, $\tilde x^i/\tilde
\pi^i(x)$ and $\tilde \Gamma^i(x)$ denote respectively the vectors
$[x^i_v]_{v \in \Supp{x^i}}$, $[\tilde \pi^i_v(x)]_{v \in
\Supp{x^i}}$, $[x^i_v/\pi^i_v(x)\big)]_{v \in \Supp{x^i}}$, and the
matrix $[\Gamma^i_{wv}(x)]_{w,v \in \Supp{x^i}}$, where $\Supp{x^i} =
\{v \in V^i \mid x_v^i > 0\}$.

The proof proceeds hereafter along two main steps.

\textit{\textbf{Step 1}}. We show first that there are functions
$\quasegrd^i : \D \to (0, \infty)$, $i \in [m]$, such that
\begin{equation}\label{eqn:Lyap&entropy}
  \big\langle \nabla L(x), F(x)\big\rangle
  =
  \sum_{i \in [m]} \quasegrd^i(x) \Big\langle
  \varphi\big(\tilde x^i/\tilde \pi^i(x)\big),\  \tilde x^i \tilde
  \Gamma^i(x) \Big\rangle \quad \text{for all} \quad x \in \D.
\end{equation} 

For each $i \in [m]$ and $v\in V$, direct derivation shows that 
\[
  \pi_v^i(x)
  =
  \frac{x_v^i(\partial L(x)/\partial x_v^i\big)^\alpha}{N^i(x)} 
\]
where $L$ is the Lyapunov function elicited in
Theorem~\ref{th:Lyapunov} and $N^i(x)$ is the denominator of
\eqref{eqn:pi}.  Now, let $F^i(x)$ be the $i$-th coordinate
function of the vector field $F$ in \eqref{eqn:THE_field}, that is
$F^i(x) = [F_v^i\big( x \big)]_{v \in V^i}$, and denote by $\tilde
F^i(x) =$ $[F_v^i\big( x \big)]_{v \in \Supp{x^i}}$.  To simplify
notation, in the sequel we will write $N^i(x)$ for the denominator of
\eqref{eqn:pi}. Noting that $N^i(x) > 0$ for all $x\in \D$ and $i \in
[m]$, and taking into account that $\pi_v^i(x) > 0$ for $x_v^i > 0$,
$F_v^i(x) = 0$ for $x_v^i = 0$ and $|\partial L(x)/\partial x^i_v| <
\infty$ for all $x \in \D$, $i \in [m]$, $v \in V^i$, gives
\begin{align*}
  \big\langle \nabla L(x), F(x)\big\rangle
  &=
    \sum_{i\in [m]}\sum_{v\in V^i} -\frac{\partial L(x)}{\partial 
    x^i_v}F_v^i(x)                                                  \\
  &=
    \sum_{i\in [m]}\sum_{v\in \Supp{x^i}} -\frac{\partial L(x)}{\partial 
    x^i_v}F_v^i(x)                                                  \\
  &=
    \sum_{i\in [m]} \big(N^i(x)\big)^{1/\alpha} \sum_{v\in \Supp{x^i}}
    -\bigg(\frac{x_v^i}{x_v^i \big(\partial L(x)/\partial
    x_v^i\big)^\alpha} N^i(x)\bigg)^{-1/\alpha}F_v^i(x)             \\
  &=
    \sum_{i\in [m]} \big(N^i(x)\big)^{1/\alpha} \sum_{v\in \Supp{x^i}}
    - \big(x_v^i/\pi_v^i(x)\big)^{-1/\alpha} F_v^i(x)               \\
  &=
    \sum_{i\in [m]} \big(N^i(x)\big)^{1/\alpha}
    \Big\langle \varphi\big(\tilde x^i/\tilde \pi^i(x)\big), \tilde
    F^i(x)\Big\rangle.
\end{align*}
To conclude set
\[
  \quasegrd^i(x) = N^i(x)^{1/\alpha}
\]
and note that that $\tilde F^i(x) = \tilde x^i \tilde \Gamma^i(x)$ as
$F^i(x) = x^i \Gamma^i(x)$ and $x_v^i = 0 \implies \Gamma_{w v}^i(x) =
0$ for all $x \in \D$, $i \in [m]$, and $w, v \in V^i$.

\textit{\textbf{Step 2}}. 
We will show that $\big\langle \nabla L(x), F(x)\big\rangle = 0$ if $x
\in \Fix(\pi)$ and $\big\langle \nabla L(x), F(x)\big\rangle < 0$ if
$x \in \D {\setminus} \Fix(\pi)$.  Since $F(x) = 0$ if and only if $x
\in \Fix(\pi)$ and since $\quasegrd^i(x) > 0$ for all $x \in \D$, it
suffices to show that $\big\langle \varphi\big(\tilde
x^i/\tilde\pi^i(x)\big),\ \tilde x^i \tilde \Gamma^i(x) \big\rangle
\leq 0$ for all $x \in \D {\setminus} \Fix(\pi)$ and all $i \in [m]$,
where the previous inequality holds strictly for at least one
particular $i \in [m]$, possibly depending on $x \in \D {\setminus}
\Fix(\pi)$.  Let $x \in \D {\setminus} \Fix(\pi)$ and $i \in [m]$ be
fixed but arbitrary. Observe now that $\tilde \Gamma^i(x)$ is
irreducible because $\tilde\pi_v^i(x) > 0$. Further,  $\tilde
\pi^i(x)$ is the unique invariant probability measure of $\tilde
\Gamma^i(x)$. In this case, Lemma 4 in \cite[page 128]{B15} may be used 
to get the bound 
\begin{equation}\label{eqn:lemma4}
  \Big\langle  \varphi\big(\tilde x^i/\tilde\pi^i(x)\big),\ \tilde x^i
  \tilde \Gamma^i(x) \Big\rangle
  \leq
  - \lambda^i(x) \inf_{v \in \Supp{x^i}}
  \varphi'\big(x_v^i/\pi_v^i(x) \big) \sum_{v \in \Supp{x^i}}
  \frac{(x_v^i - \pi_v^i(x))^2}{\pi_v^i(x)},
\end{equation}
where $\lambda^i(x) = \lambda(\tilde \Gamma^i(x)) > 0$ is the spectral gap
of $\tilde \Gamma^i(x)$.  To conclude, note that the right hand side of
\eqref{eqn:lemma4} is less or equal than zero because $\varphi'(z) =
\alpha^{-1} z^{-(\frac{1}{\alpha} + 1)} > 0$ for all $z > 0$ and
$\alpha > 0$. Moreover, if the right hand side of \eqref{eqn:lemma4}
were zero for all choices of $i \in [m]$, it must be the case that
$x_v^i - \pi_v^i(x) = 0$ for all $v \in V^i$ and $i \in [m]$, which is
impossible as by assumption $x \notin \Fix(\pi)$.
\end{proof}

The following lemma is a version of Proposition 3.2 in \cite{B96}
adapted to our needs.

\begin{lemma}\label{lem:LEq_is_finite}
Let $\Lambda \subset \D$ be a compact and forward invariant set for
$\Phi$. Assume that $L: \D\to\R$ is a Lyapunov function for $\Lambda$
such that the cardinal of $L(\Lambda)$ is finite. Then
$
  \CR \subset \Lambda.
$
\end{lemma}

\begin{lemma}\label{lem:L_is_constant}
Let $L$ be as defined in \eqref{eqn:Lyapunov-pwr-H}, then
$|L(\Fix(\pi))| < \infty$.
\end{lemma}
\begin{proof}
Recall that $\Fix(\pi) = \cup_{S \in \V} \D(S)$. Let $S \in \V$ and
$x, y \in \D(S)$ be arbitrary. Since $\V$ is finite, it is sufficient
to show that $L(x) = L(y)$. To show this, we set $z(t) = tx + (1-t)y$
for $t \in [0,1]$ and show that derivative of $t \mapsto L(z(t))$ is
zero. Note that, since $\D(S)$ is convex, $z(t) \in \D(S) \subset
\Fix(\pi)$. As a consequence, it follows that $z_v^i = \pi_v^i(z) =
z_v^i H_v^i(z)/N^i(z)$, where $N^i(z)$ does not depend on $v$. Thus,
for $v \in S^i$, we have $z_v^i > 0$ and so $H_v^i(z) = N^i(z)$ for
all $v \in S^i$. Note that $\partial L(z)/\partial z_v^i =
-(H_v^i(z))^{1/\alpha} = -(N^i(z))^{1/\alpha}$, thus the derivative of
$L$ along the direction of $z(t)$ is
\[
\frac{d L(z(t))}{dt} 
 =
 \sum_{i \in [m], v \in S^i} \frac{\partial L(z)}{\partial
   z_v^i}\frac{d z_v^i(t)}{dt} \\ 
 = - \sum_{i \in [m]} \big(N^i(z(t))\big)^{1/\alpha}  \frac{d}{dt}
 \sum_{v \in S^i}  z_v^i(t) = 0. 
\]
The last equality follows because $z^i(t) \in \triangle^i$ and
$z_v^i(t) = 0$ for all $v \notin S^i$ and $t \in [0,1]$. 
\end{proof}

The assertion ($iv$) in Theorem~\ref{th:det} follows by taking
$\Lambda = \Fix(\pi)$ and using Lemmas \ref{lem:limitset},
\ref{lem:LEq_is_finite} and \ref{lem:L_is_constant}. 

\subsubsection{Proof of assertion $(v)$ of Theorem~\ref{th:det}}

Note first that
if $S\in \V$ is such that $\det R_\rpar(S) \neq 0$, then the
associated system \eqref{eqn:THE_system} has a unique solution,
i.e. $|\D(S)| = 1$. Now, if $\det R_\rpar(S) \neq 0$ for any $S \in
\V$, it then follows by Theorem~\ref{th:det}.(i) and by the fact that
$\V$ is finite that $\Fix(\pi)$ is finite. $\Fix(\pi)$ is clearly
forward invariant by $\Phi$ as from \eqref{eqn:THE_field} it is
constituted by points $x \in \D$ such that $F(x) = 0$. Let $L$ be the
strict Lyapunov function in Theorem~\ref{th:Lyapunov}. Setting
$\Lambda = \Fix(\pi)$ in Lemma~\ref{lem:LEq_is_finite} and observing
that $|L(\Fix(\pi))| < \infty$ shows that almost surely $\CR \subset
\Fix(\pi)$. This, combined with the second assertion of
Lemma~\ref{lem:limitset}, gives $\Prb\big(\limset \subset
\Fix(\pi)\big) = 1$ and concludes the proof of assertion $(iv)$ of
Theorem~\ref{th:det}.

\subsection{Proof of Corollaries~\ref{cor:conv_is_generic} and
  \ref{cor:emptyintesection}} 

\begin{proof}[Proof of Corollary~\ref{cor:conv_is_generic}]
Let $\V$ be given by  $\G$ and consider $S \in \V$ fixed but
arbitrary. Let $\eta = (\eta_v^i)_{v, i}$ and $\rpar =
(\rpar_v^{ij})_{v,i,j}$ for $i, j \in [m]$ and $v \in V^i$. Define
$\Theta = \big\{(\alpha, \eta, \rpar) \mid \alpha > 0, \eta_v^i > 0,
\rpar_v^{ij} \in \R \text{ satisfying \eqref{eqn:rho} and
\eqref{eqn:eta}}\big\}$ that is, $\Theta$ is the set of possible
values for $\alpha$, $\eta_v^i$ and $\rpar_v^{ij}$. Observe that
$\Theta \subset \R_+ \times \R_+^d \times \R^r$ where $\R_+ =
(0,\infty)$ and $r = d(d+1)/2$. Consider now the set of $d$-by-$d$
symmetric matrices with real entries, $M_{d,d}(\R)$, and then the
bijection $f : \Theta \to M_{d,d}(\R)$ defined by $\rpar \mapsto
R_\rpar(S)$, where $R_\rpar(S)$ is the matrix of coefficients of the
system \eqref{eqn:THE_system}. Let $\Theta_0(S) \subset \Theta$ be
defined as $\Theta_o(S) = \big\{(\alpha, \eta, \rpar) \in \Theta \mid
\det f(\rpar) = 0\big\}$, the set of parameters for which $\det
R_\rpar(S) = 0$.  As the set of matrices $M_{d,d}(\R)$, seen as a
subset of $\R^r$, that are singular, is a Lebesgue null set, it follows
that $\Theta_o(S)$ is also a null set. Thus, for almost all $\rpar_v^{ij}$
we have that $\det R_\rpar(S) \neq 0$. As this holds for each $S \in
\V$ and $\V$ is finite, the claim follows from item $(iv)$ in
Theorem~\ref{th:det}.
\end{proof}

\begin{proof}[Proof of Corollary~\ref{cor:emptyintesection}] 
Since $\Supp{x^i} \cap \Supp{x^j} = \varnothing$ for all $i, j \in
[m]$ such that $i \neq j$, and $\eta_v^i = \eta_w^i$ for all $v, w \in
\Supp{x^i}$, we have that
\eqref{eqn:THE_system} reduces to
\begin{equation}\label{eqn:linear-reduced}
   \forall i \in [m], \, \forall v, w \in \Supp{x^i}: \,\,
   \rpar_v^{ii} x_v^i  =  \rpar_w^{ii} x_w^i. 
\end{equation}

Suppose that \eqref{enu:C1} or \eqref{enu:C2} holds for each $i \in
[m]$. To show that $x \in \Fix(\pi)$, we apply Theorem \ref{th:det}
and show that Equation \eqref{eqn:linear-reduced} is satisfied. Fix $i
\in [m]$. If \eqref{enu:C1} is satisfied, then $\rpar_v^{ii} x_v^i = 0 =
\rpar_w^{ii} x_w^i$ for all $v, w \in \Supp{x^i}$. If \eqref{enu:C2} is
satisfied then $\rpar_v^{ii} x_v^i = 1/ \big (\sum_{k \in \Supp{x^i}}
1/\rpar_k^{ii} \big) = \rpar_w^{ii} x_w^i$ for all $v, w \in
\Supp{x^i}$.

Suppose now that $x \in \Fix(\pi)$. It follows by Theorem \ref{th:det}
that Equation \eqref{eqn:linear-reduced} holds. Next we show that, for
each $i \in [m]$, \eqref{enu:C1} or \eqref{enu:C2} holds. Fix $i \in
[m]$. If \eqref{enu:C1} satisfied we are done. Assume now that
condition \eqref{enu:C1} is not satisfied. Then there is $w \in
\Supp{x^i}$ such that $\rpar_w^{ii} \neq 0$ and it follows by Equation
(\ref{eqn:linear-reduced}) that $\rpar_v^{ii} = \rpar_w^{ii} x_w^i
\big /x_v^i$ for all $v \in \Supp{x^i}$. Since $x_w^i \big /x_v^i > 0$
for all $v \in \Supp{x^i}$, we have that all $\rpar_v^{ii}$, $v \in
\Supp{x^i}$, have the same sign as $\rpar_w^{ii} \neq 0$, that is, all
$\rpar_v^{ii}$, $v \in \Supp{x^i}$, are positive or all are negative.
Further, from (\ref{eqn:linear-reduced}), we have that
$x_w^{i}/x_v^{i} = \rpar_v^{ii} \big /\rpar_w^{ii}$ for all $v, w \in
\Supp{x^i}$. Thus, summing both sides over $w \in \Supp{x^i}$ and
taking into account the identity $\sum_{w \in \Supp{x^i}} x_w^i = 1$,
gives $1/x_v^{i} = \sum_{w \in \Supp{x^i}}
(1/\rpar_w^{ii})/(1/\rpar_v^{ii})$ for all $v \in \Supp{x^i}$. Observe
that this is equivalent to \eqref{enu:C2} and so the proof is
concluded.
\end{proof}

\section{Proofs of Theorems~\ref{th:no-toface} and
  \ref{th:no-tointerior}}
The proof of Theorem~\ref{th:no-toface} is inspired by the proof of
Theorem 1.3 \cite{P92} which concerns the non-convergence of a single
self reinforced random walk to points that are on the boundary of the
$(d-1)$-simplex $\triangle$. The proof of Theorem~\ref{th:no-toface}
evokes lemmas \ref{Coupling-Working} and \ref{lem:b.r.v}. Both these
lemmas are stated in this section but their proofs are presented in
the Appendix. We need to introduce the following definitions. For any
$z\in \R$, let $\floor*{z}$ be the greatest integer less than or equal
to $z$. For any $\delta > 0$, and sufficiently large $n \in
\mathbb{N}$ such that $n \leq \floor*{n(1+\delta)} - 1$, let
\begin{equation}
\label{eqn:Tdef}
  T_{n,\delta} = \big\{n, n+1, \ldots, \floor*{n(1+\delta)}
  - 1 
  \big\}.
\end{equation}

\begin{lemma}\label{Coupling-Working}
Let $p\in \D$ be a point satisfying the condition \eqref{eqn:abar}.
If $\Prb\big(\lim_{n\to\infty} X(n) = p\big) > 0$, then there are
$\delta >0$, $b>0$, $n_0 > 0$, and $A\in \F$ with $\Prb(A)>0$, such
that
\begin{equation}
 \label{eqn:b1bber0}
 \Prb\Big\{\Prb\big(W^i(m+1) = k \mid \F_m  \big)
 \geq X^i_k(n)(1+b)\ \Big|\ A\Big\} = 1 \quad \text{for all }\ m \in 
  T_{n,\delta} \ \ \text{and}\  \ n > n_0.
\end{equation}
\end{lemma}

\begin{lemma}\label{lem:b.r.v}
Let $\Ber(n) \in \{0,1\}$, $n = 1,2, \ldots$ be Bernoulli random
variables defined on a probability space $(\Omega, \F, \Prb)$ endowed
with a filtration $\F_n$, not necessarily generated by
$\Ber(n)$. Assume that $\Ber(n)$ is adapted to $\F_n$ and set $\bar
\Ber(n) = \big(1+\sum_{\ell=1}^n \Ber(\ell) \big)/(q+n)$. Then, there
are no $\delta >0$, $b>0$, $n_0 > 0$, $q > 0$, and $A \in \F$ with
$\Prb(A)>0$ such that event
\begin{equation}
 \label{eqn:b1bber}
 \Prb\Big\{\Prb \big(\Ber(m+1) = 1 \mid \F_m \big) \geq \bar
  \Ber(n)(1+b)\ \Big|\ A\Big\} = 1
  \quad \text{for all }\ m \in
  T_{n,\delta} \ \ \text{and}\ \ n > n_0.
\end{equation}
\end{lemma}

\begin{proof}[Proof of Theorem~\ref{th:no-toface}]
Let $p\in \D$ be a point satisfying \eqref{eqn:abar}. Suppose by
contradiction that $\Prb\big(\lim_{n\to\infty} X(n) = p\big) >
0$. According to Lemma \ref{Coupling-Working}, there are $\delta >0$,
$b>0$, $n_0 > 0$, and $A\in \F$ with $\Prb(A)>0$, such that the
condition in \eqref{eqn:b1bber0} holds. Let $\Ber(n) = \Ind\{W^i(n) =
k\}$, $\bar \Ber(n) = \big(1 + \sum_{\ell=1}^n \Ber(\ell)
\big)/(q+n)$, and $q = d^i$.  In this case, both conditions
\eqref{eqn:b1bber0} and \eqref{eqn:b1bber} are the same. This leads to
a contradiction because according to Lemma \ref{lem:b.r.v}, there are
no $\delta >0$, $b>0$, $n_0 > 0$, and $A\in\F$ with $\Prb(A)>0$ such
that the condition \eqref{eqn:b1bber} holds. This completes the proof
of the Theorem.
\end{proof}

The proof of Theorem~\ref{th:no-tointerior} relies on Theorem~1 in
\cite{P90}, which for completeness we present adapted to our needs
as Theorem~\ref{th:pemantle_interior} below. We need the
following definition.

\begin{definition}[Linearly unstable point]\label{def:linear_unstable}
Let $F$ be the vector field in \eqref{eqn:THE_field} and $p \in \D$ be
such that $F(p)=0$. Let $\JacF{p}$ be the Jacobian matrix of $F$ at
$p$ and denote by $\sigma\big(\JacF{p}\big)$ its spectrum. The point
$p$ is said to be hyperbolic if $\mathfrak{R}(\beta) \neq 0$ for all
$\beta \in \sigma\big(\JacF{p}\big)$. An hyperbolic point $p$ is said
to be linearly unstable if there is at least one $\beta \in
\sigma\big(\JacF{p}\big)$ such that $\mathfrak{R}(\beta) > 0$.
\end{definition}

\begin{theorem}[R. Pemantle, \cite{P90}]\label{th:pemantle_interior}
Let $X = \{X(n); n \geq 0\}$ be the process satisfying
(\ref{eqn:SA}) with $\E[U(n)\mid\F_n] = 0$ and such that $\D$ is
invariant for the vector field $F$ in \eqref{eqn:THE_field}. Let $p
\in \D^\circ$ be linearly unstable and let $\Bhood(p)$ be a
neighbourhood of $p$. Assume that there are constants $c_1$, $c_2 > 0$
for which the following two conditions are satisfied whenever $n$ is 
sufficiently large:
\begin{enumerate}[$(h_\bgroup1\egroup)$, nosep]
\item\label{enu:h1} $\E\big[\, \max\{\langle\theta, U(n)\rangle,
  0\}\mid X(n) = x\big] \geq c_1$ for all $x \in \Bhood(p)$ and
  $\theta \in \TsD$ with $\|\theta\|=1$
\item\label{enu:h2} $\|U(n)\| \leq c_2$ 
\end{enumerate} Assume that the vector field $F$ is at least $C^2$ so
that the stable manifold theorem holds. In this case
\[
  \Prb\Big(\lim_{n\to\infty} X(n) = p\Big) = 0.
\]
\end{theorem}

\begin{proof}[Proof of Theorem~\ref{th:no-tointerior}]
Let $p \in \D^\circ$ be a fixed point of $\pi$. For the vector
field $F$ in \eqref{eqn:THE_field} it then follows that $F(p) = 0$.
Further 
\begin{align*}
    \det\big((DF(p) - (\lambda - 1)I\big) = 0
  &\ \ \Leftrightarrow \ \ 
  \det\big((D\pi(p) - I) - (\lambda - 1)I\big) = 0 \\
  &\ \ \Leftrightarrow \ \ 
    \det\big(D\pi(p) - \lambda I\big) = 0,
\end{align*}
showing that $\lambda$ is an eigenvalue of $D\pi(p)$ if, and only if
$\lambda - 1$ is an eigenvalue of $\JacF{p}$. Taking into account the
hypotheses for $\lambda$ in the statement of the theorem and setting
$\beta = \lambda -1$ in Definition~\ref{def:linear_unstable} shows
that $p$ is linearly unstable.

From \eqref{eqn:THE_field} and \eqref{eqn:pi} it is relatively simple
to observe that $F$ is $C^2$. The rest of the proof consists in
verifying the conditions \ref{enu:h1} and \ref{enu:h2} of
Theorem~\ref{th:pemantle_interior}. \ref{enu:h2} is immediate from the
definition of $U(n)$ in Lemma~\ref{lem:SA}. The verification of
\ref{enu:h1} follows along the lines of the proof of Lemma 9 in
\cite{RPP22} which rests on the following two facts. First, the walks
$W^i(n+1)$, $i\in [m]$, are independent given $\F_n$. Second, $\pi(x)
>0$ holds for any $x$ in a sufficiently small neighbourhood of $p$ in
$\D^\circ$. The former condition is satisfied in our case by the
definition of the process $W(n)$, as observed in the paragraph bellow
of equation \eqref{eqn:eta}. The latter follows in view of the
definition of $\pi$ in \eqref{eqn:pi}, taking into account
\eqref{eqn:eta}.
\end{proof}

\section{Proofs of the results for complete
  graphs}\label{sec:ProfCompletGraphs} 
We start with the following lemma which will be used in the proof of
Theorem~\ref{th:complete_G}. We recall the notation from
\eqref{eqn:sizes}, namely, for $p = (p^1, p^2) \in \D$ let $U(p) =
\Supp{p^1}\cap \Supp{p^2}$ and $u(p) = |U(p)|$.

\begin{lemma}\label{lem:samesupport}
Suppose that \eqref{eqn:rhoij} and \eqref{eqn:etaiv} hold with
$\epsilon = 0$. Let $p = (p^1, p^2)\in \D$.  Suppose
$\Supp{p^1}\cap\Supp{p^2} = \varnothing$, then $p \in
\Fix(\pi)$. Suppose $\Supp{p^1}\cap\Supp{p^2} \neq \varnothing$, then
$p \in \Fix(\pi)$ if and only if $\Supp{p^1}=\Supp{p^2}$ and $p_v^i =
1/u(p)$, $i=1, 2$, for all $v \in \Supp{p^1}$.
\end{lemma}
\begin{proof}
Suppose first that $\Supp{p^1}\cap\Supp{p^2} = \varnothing$. The first
assertion follows because the condition \eqref{enu:C1} of
Corollary~\ref{cor:emptyintesection} is satisfied.  Now assume that
$\Supp{p^1}\cap\Supp{p^2} \neq \varnothing$. Then, if $\Supp{p^1} =
\Supp{p^2}$ and $p_v^i = 1/u(p)$ for all $v \in \Supp{p^i}$, clearly
it follows that $p_v^i = \pi_v^i(p)$ for all $v \in V$ and $i \in
[2]$, that is $p \in \Fix(\pi)$. Suppose, now by contradiction, $p \in
\Fix(\pi)$, $\Supp{p^1} \cap \Supp{p^2} \neq \varnothing$ and
$\Supp{p^1} \neq \Supp{p^2}$. Then, there are $v, w \in V$, $i \in
[2]$ and $j = 3-i$, such that $v, w \in \Supp{p^i}$, $v \in
\Supp{p^j}$, and $w \notin \Supp{p^j}$. Since $v, w \in \Supp{p^i}$ it
follows from \eqref{eqn:THE_system} that $p_v^j = p_w^j$.  This leads
however to the following contradiction: $p_v^j > 0$ and $p_v^j = 0$,
where the last equality holds because $p_v^j = p_w^j$ and $p_w^j = 0$
as $w \notin \Supp{p^j}$. Now, assuming that $p \in \Fix(\pi)$ and
$\Supp{p^1} = \Supp{p^2}$, it follows by \eqref{eqn:THE_system} that
$p_v^i = p_w^i$ for all $v,w \in \Supp{p^i}$. This is only possible if
$p^i$ is the uniform probability measure over $\Supp{p^i}$, that is,
$p_v^i = 1/u(p)$ for all $v \in \Supp{p^i}$.
\end{proof}

\begin{proof}[Proof of Theorem~\ref{th:complete_G}] 
The set of fixed points of $\pi$, $\Fix(\pi)$, can be described as an
union indexed by the subsets $U\subset V$, that is
\[
  \Fix(\pi) = \bigcup_{U \subset V} \Fix_U(\pi),
  \quad\text{where}\quad
  \Fix_U(\pi) = \Big\{p \in \Fix(\pi) \,\Big|\, \Supp{p^1} \cap
  \Supp{p^2} = U\Big\}.
\]

If $U \neq \varnothing$ it follows by Lemma \ref{lem:samesupport} that
$\Fix_U(\pi)$ corresponds to the singleton set $\{p = (p^1, p^2)\}$ in
which for each $i = 1, 2$, $p_v^i = 1/u(p)$ when $v \in U$, and $p_v^i
= 0$ for $v \notin U$, that is, $p$ is the point in which $p^i$ is
the uniform probability measure over $U$, $i=1,2$.

We will show next that the limit set $\limset$ cannot be
a subset of $\Fix_U(\pi)$ when $U\neq \varnothing$.  Suppose by
contradiction that $\Prb (\limset \subset \Fix_U(\pi)) >
0$ for some $U \neq \varnothing$. Assume first that $U \neq V$, that
is, there is at least one $k \in V$ with $k \notin U$. Then, according
to Lemma~\ref{lem:samesupport}, there is a unique element $p$ of 
$\Fix_U(\pi)$ such that for $i=1,2$, $p_v^i = 1/u(p)$ for $v \in U$
and $p_v^i = p_k^i = 0$ for $v \notin U$. As a consequence
\begin{align*}
  \lim_{x \to p} \frac{\pi_k^1(x)}{x_k^1}
  &=
    \lim_{x \to p} \frac{1}{x_k^1} \frac{x_k^1(\eta -
    x_k^2)^\alpha}{\sum_{v \in V} x_v^1(\eta -
      x_v^2)^\alpha}             \\ 
  &=
    \frac{\eta^\alpha}{\sum_{v \in U} p_v^1 (\eta -
      1/u(p))^\alpha} =  
    \bigg(\frac{\eta}{\eta - 1/u(p)}\bigg)^\alpha > 1,
\end{align*}
where $\eta = \eta_v^i$ for all $v \in V$, $i \in \{1, 2\}$. By
Theorem~\ref{th:no-toface} it follows that $\Prb\big(\lim_{n \to
\infty} X(n) = p \big) = 0$. Because $p$ is the unique element of
$\Fix_U(\pi)$, we have that $\Prb (\limset \subset \Fix_U(\pi)) = 0$.
 
Assume now that $U = V$. In this case $\Fix_U(\pi) = \{p\}$, where $p$
is the point with coordinates equal to $1/\kappa$ with $\kappa =
|V|$. We will show that $\Prb \big (\lim_{n \to \infty} X(n) = p \big
) = 0$. Straightforward computations show that the Jacobian matrix
$D\pi(p)$ has the form
\[
  \frac{1}{\kappa}\begin{bmatrix}
     \kappa I_\kappa -\mathbf{1}_{\kappa}   &
     \big(\mathbf{1}_{\kappa} - \kappa  
   I_{\kappa}\big)\displaystyle\frac{\alpha}{2\kappa-1}\\[1em]  
     \big(\mathbf{1}_{\kappa} - \kappa
     I_{\kappa}\big)\displaystyle\frac{\alpha}{2\kappa-1}  
                     & \kappa I_\kappa -\mathbf{1}_{\kappa}
   \end{bmatrix},
\] 
where $\mathbf{1}_{\kappa}$ is a $\kappa$-by-$\kappa$ matrix of ones
and $I_\kappa$ is the $\kappa$-by-$\kappa$ identity matrix. For all
$\kappa\geq 2$ and $\alpha > 0$, the eigenvalues of this matrix have
real parts different from 1 and, up to algebraic multiplicity,
one eigenvalue equals
\begin{equation}\label{eqn:center_eigenval}
  \frac{\alpha}{2\kappa-1} + 1.
\end{equation}
This shows that $p$ is hyperbolic and linearly unstable for all
$\kappa\geq 2$ and $\alpha > 0$. Since $p$ belongs to the relative
interior of $\D$, the direct application of
Theorem~\ref{th:no-tointerior} shows that $X$ has zero probability of
converging to $p$.

To conclude the proof, it only remains to consider the case $U =
\varnothing$. From Lemma \ref{lem:samesupport} it follows that
$\Fix_\varnothing(\pi) = \eK$ where $\eK$ is defined in \eqref{eqn:K}.
The set $\Fix_\varnothing(\pi)$ is disconnected from the sets
$\Fix_U(\pi)$, $U\neq \varnothing$. Since $\Prb\big(\limset\subset
\Fix_U(\pi)\big) = 0$ for $U\neq \varnothing$, it then follows from
Theorem~\ref{th:det}.$(iii)$ and $(iv)$ that $\Prb\big(\limset
\subset \Fix_\varnothing(\pi)\big) = 1$.
\end{proof}

\begin{proof}[Proof of
  Corollary~\ref{cor:overlap_complete_G_eps_zero}]
This is an immediate consequence of the fact that $X$ accumulates
almost surely along $\eK$ in \eqref{eqn:K} and the definition of
$\ducal H$ in \eqref{eqn:H}.
\end{proof} 

The next two lemmas will be used in the proof of
Theorem~\ref{th:complete_G_epsilon}.

\begin{lemma}\label{lem:uniform_epsilon}
Suppose that \eqref{eqn:rhoij} and \eqref{eqn:etaiv} hold with
$\epsilon > 0$.  Let $p=(p_1, p_2) \in \Fix(\pi)$. If $\epsilon$ is
sufficiently small, then
\begin{enumerate}[$(i)$, nosep]
  \item  $p_v^i= p_w^i$ for all $v, w \in U$, $i=1,2$;
  \item  $p_v^i= p_w^i$ for all $v, w \in U^i$, $i=1,2$.
\end{enumerate}
\end{lemma}

\begin{proof}
To prove the first assertion, suppose $v, w \in U$ then $v, w \in
\Supp{p^1} \cap \Supp{p^2}$. Equation \eqref{eqn:THE_system} in this
case gives $\epsilon p_v^1 + p_v^2 = \epsilon p_w^1 + p_w^2$ and
$\epsilon p_v^2 + p_v^1 = \epsilon p_w^2 + p_w^1$, and so $(\epsilon^2
- 1) p_w^1 = (\epsilon^2 - 1) p_v^1$. For $\epsilon < 1$, this gives
$p_v^1 = p_w^1$. Similarly one can deduce that $p_v^2 = p_w^2$.  To
prove the second assertion, suppose that $v, w \in U^1$. In this case
$v, w \in \Supp{p^1}$ and $v, w \notin \Supp{p^2}$. Equation
\eqref{eqn:THE_system} then gives $\epsilon p_v^1 = \epsilon
p_w^1$. Since $\epsilon > 0$ we have that $p_v^1 = p_w^1$. Arguing in
the same fashion also shows that $p_v^2 = p_w^2$.
\end{proof}

\begin{lemma}\label{lem:alleq}
Suppose that \eqref{eqn:rhoij} and \eqref{eqn:etaiv} hold with
$\epsilon > 0$.  Set $\eK_2^c = \big \{ \, p
\in \eK^c | \ u^1 = u^2 = 0 \big \}$ and define
\[
  \tilde{\eK}_2^c = \bigg\{ \, p \in \eK_2^c\ 
  \Big| \
  p_v^i = \frac{1}{s^i} \,\, \text{ for } \,\, v \in \Supp{p^i}\, 
  \text{ and } \, i=1,2 \bigg\},
\]
that is, $p \in \tilde{\eK}_2^c$, if and only if $p^1 = p^2$, where
$p^2$ is uniformly distributed over its support. In particular,
$\tilde{\eK}_2^c$ includes $p$, with $p_v^1 = p_v^2 = 1/d$ for all $v
\in V$.  Then, for sufficiently small $\epsilon > 0$, it follows that
\[
  \Fix(\pi) = \tilde{\eK} \cup \tilde{\eK}_1^c \cup
  \tilde{\eK}_2^c
\]
where $\tilde{\eK}$ and $\tilde{\eK}_1^c$ are defined as
in Theorem~\ref{th:complete_G_epsilon}.
\end{lemma}

\begin{proof}
Let $\eK_1^c$ be defined as in Theorem~\ref{th:complete_G_epsilon} and
set $\eK_3^c = \big \{ \, p \in \eK^c | \ u^i > 0 \, \text{ and } \,
u^{3-i} = 0 \, \text{ for some } i =1,2 \big \}$.  Note that $\eK^c =
\eK_1^c \cup \eK_2^c \cup \eK_3^c$. It suffices to show that, for
sufficiently small $\epsilon > 0$, the following assertions hold: $a)$
$\Fix(\pi) \cap \eK = \tilde{\eK}$, $b)$ $\Fix(\pi) \cap \eK_1^c =
\tilde{\eK}_1^c$, $c)$ $\Fix(\pi) \cap \eK_2^c = \tilde{\eK}_2^c$, and
$d)$ $\Fix(\pi) \cap \eK_3^c = \emptyset$.

Proof of assertion $a)$: $\Fix(\pi) \cap \eK = \tilde{\eK}$. Since
$\tilde{\eK} \subset \eK$, it is sufficient to show that, if $p \in
\eK$, then $p \in \Fix(\pi)$ if and only if $p \in
\tilde{\eK}$. Assuming $p \in \eK$, that is, $\Supp{p^1} \cap
\Supp{p^2} = \emptyset$, and applying Corollary
\ref{cor:emptyintesection} \eqref{enu:C2} under the assumption that
$\rpar^{ii}_v = -\epsilon$, it follows that $p \in \Fix(\pi)$ if and
only if $p^i_v = 1/s^i$ for all $v \in \Supp{p^i}$.

Proof of assertion $b)$ $\Fix(\pi) \cap \eK_1^c = \tilde{\eK}_1^c$:
Suppose first that $p \in \Fix(\pi) \cap \eK_1^c$. In this case, from
Lemma~\ref{lem:uniform_epsilon}, there are numbers $0\leq \delta^i,
\tilde\delta^i \leq 1$, $i=1,2$, 
that $p_v^i = \delta^i$ if $v \in U$ and $p_w^i = \tilde\delta^i$ if
$w \in U^i$.  Since $\Supp{p^i} = U \cup U^i$ and $U \cap U^i =
\emptyset$, it follows that $\sum_{v \in U} p_v^i + \sum_{w \in U^i}
p_w^i = 1$, that is, $u \delta^i + u^i \bar \delta^i =1$. Let $v \in
U$ and $w \in U^i$. Since $U^i \cap \Supp{p^j} = \emptyset$, it
follows from \eqref{eqn:THE_system} that $\epsilon p_v^i + p_v^j =
\epsilon p_w^i$, that is, $\epsilon \delta^i + \delta^j = \epsilon
\bar \delta^i$ for $i = 1,2$ and $j = 3 -i$. Summing up, $p\in
\tilde{\eK}_1^c$ defined by some $\delta^i$ and $\bar\delta^i$
satisfies \eqref{eqn:THE_system} if, and only if $\delta^i$ and $\bar
\delta^i$ satisfy the system: $u \delta^i + u^i \bar \delta^i =1$ and
$-\epsilon \delta^i - \delta^j = -\epsilon \bar \delta^i$, $i = 1,2$
and $j = 3 - i$. Using the fact that $s^i = u^i + u$, , $i = 1,2$, it
is straightforward to see that this system has a unique solution in
the variables $\delta^i$ and $\bar\delta^i$, $i = 1,2$, given by
\eqref{eqn:deltaEbardelta} and thus that $p \in
\tilde\eK_1^c$. Suppose now that $p \in \tilde\eK_1^c$. Then clearly
$p \in \eK_1^c$ and, as shown previously, $p$ is a solution of
\eqref{eqn:THE_system}, thus $p \in \Fix(\pi) \cap \eK_1^c$.

Poof of assertion $c)$: $\Fix(\pi) \cap \eK_2^c =
\tilde{\eK}_2^c$. Since $\tilde{\eK}_2^c \subset \eK_2^c$, it is
sufficient to show that, if $p \in \eK_2^c$, then $p \in \Fix(\pi)$ if
and only if $p \in \tilde{\eK}_2^c$. Assume $p \in \eK_2^c$, and so
that $\Supp{p^1} = \Supp{p^2} = U$. If $p \in \Fix(\pi)$, then by
Lemma~\ref{lem:uniform_epsilon}.($i$), $p \in \tilde{\eK}_2^c$. Now
suppose $p \in \tilde{\eK}_2^c$. Clearly $p \in \eK_2^c$ and, since
$p$ satisfies \eqref{eqn:THE_system}, we have that $p \in \Fix(\pi)
\cap \eK_2^c$.

Proof of assertion $d)$: $\Fix(\pi) \cap \eK_3^c =\emptyset$. The
proof proceeds by contradiction, assuming that, for any $\varepsilon >
0$, there is a $p \in \Fix(\pi) \cap \eK_3^c$. For such a $p$ we have
that $\Supp{p^1} \neq \Supp{p^2}$ and $\Supp{p^j} \subset \Supp{p^i}$
for some $i = 1, 2$ and $j = 3 -i$.  Since $\Supp{p^j} \subset
\Supp{p^i}$, we have that $\Supp{p^j} = U$.  Then, by
Lemma~\ref{lem:uniform_epsilon}.(\emph{ii}), we have that $p_v^j =
1/s^j$ for $v \in \Supp{p^j}$. Let $v \in U$ and $w \in U^i$, then $v
\in \Supp{p^i} \cap \Supp{p^j}$, $w \in \Supp{p^i}$, and $w \notin
\Supp{p^j}$. Then, from \eqref{eqn:THE_system}, we have that $\epsilon
p_v^i + p_v^j = \epsilon p_w^i$, that is, $\epsilon p_v^i + 1/s^j =
\epsilon p_w^i$. Since $1/s^j > 0$, the previous equality is
impossible for sufficiently small $\epsilon > 0$.
\end{proof}

\begin{proof}[Proof of Theorem~\ref{th:complete_G_epsilon}]
As shown by Lemma~\ref{lem:alleq} the set of fixed points of $\pi$
equals $\Fix(\pi) = \tilde{\eK}\cup \tilde{\eK}_1^c\cup
\tilde{\eK}_2^c$. Being a finite union of finite sets, this shows that
$\Fix(\pi)$ is finite. The application of Theorem~\ref{th:det}.($iv$)
then shows that $X$ converges almost surely to an isolated point $p\in
\Fix(\pi)$. To conclude the proof, we will show that
$\Prb\big(\lim_{n\to\infty} X(n) = p\big)=0$ if $p\in
\tilde{\eK}_2^c$. We will consider separately the cases $\Supp{p^1} =
\Supp{p^2} \neq V$ and $\Supp{p^1} = \Supp{p^2} = V$.

Suppose that $p \in \tilde{\eK}_2^c$, and assume first that
$\Supp{p^1} = \Supp{p^2} \neq V$. In this case, there exists at least
one vertex $k\in V$ such that $p^i_k = 0$ for $i=1,2$. From
Lemma~\ref{lem:alleq}, for all $v \in \Supp{p^i}$, $p^i_v = 1/s$ where
$s = |\Supp{p^i}|$.  Using \eqref{eqn:pi} for any $i = 1, 2$ and
$j=3-i$ leads to
\begin{align*}
  \lim_{x \to p} \frac{\pi_k^i(x)}{x_k^i}
  &=
  \lim_{x \to p} \frac{1}{x_k^i}
  \frac{x_k^i(\eta-\epsilon x_k^i -x_k^j)^\alpha}{\sum_{w \in
      \Supp{p^i}} x_w^i(\eta - \epsilon x_w^i -x_w^j)^\alpha}  \\
  &=
    \frac{\eta^\alpha}{\sum_{w \in \Supp{p^i}}
    1/s(\eta - \epsilon/s - 1/s)^\alpha}
  =
  \bigg(\frac{\eta}{\eta - (\epsilon+1)/s}\bigg)^\alpha > 1,
\end{align*}
where $\eta = \eta_v^i$ for all $v \in V$, $i = 1, 2$.
Theorem~\ref{th:no-toface} can now be used to show that
$X$ cannot converge to any point $p \in \tilde{\eK}_2^c$ with
$\Supp{p^1} = \Supp{p^2} \neq V$.

Suppose that $p\in \tilde{\eK}_2^c$ is such that $\Supp{p^1} =
\Supp{p^2} = V$. In this case, from Lemma~\ref{lem:alleq}, all the
coordinates of $p$ are equal to $1/\kappa$, $\kappa = |V|$. By
observing that $\sigma\big(D\pi(p)\big)$, the spectrum of $D\pi(p)$,
depends continuously on $\epsilon$ and using the form of $D\pi(p)$ in
the proof of Theorem~\ref{th:complete_G}, choose $\epsilon > 0$
sufficiently small so that $\sigma\big(D\pi(p)\big)$ has an eigenvalue
arbitrarily close to \eqref{eqn:center_eigenval}. The point $p$ is thus
linearly unstable for all $\kappa\geq 2$ and $\alpha>0$ when
$\epsilon$ is sufficiently small. Since $p \in \iD$, from
Theorem~\ref{th:no-tointerior} it follows that $\Prb\big(\lim_{n \to
\infty} X(n) = p \big) = 0$.
\end{proof} 

\begin{proof}[Proof of Corollary~\ref{cor:overlap_complete_G_eps_pos}]
As the sets $\tilde{\eK}$ and $\tilde{\eK}_1^c$ in
\eqref{eqn:unifabitra} and \eqref{eqn:overlappted} are finite, from
Theorems~\ref{th:det} and \ref{th:complete_G_epsilon} it follows that
$\sum_{p \in \tilde{\eK}\cup \tilde{\eK}_1^c} \Prb\big(\lim_n
\ducal{H}(X(n)) = \ducal{H}(p)\big) = 1$. This concludes the proof of
the first assertion. To prove the second item, it is sufficient to
show that for all $p \in \tilde{\eK}\cup \tilde{\eK}_1^c$,
$\ducal{H}(p) < \kappa^3 \epsilon^2$ for sufficiently small $\epsilon
> 0$. This inequality is immediate when $p \in \tilde{\eK}$ because
for any such point,  \eqref{eqn:unifabitra} implies that 
$\ducal{H}(p) = 0$. It remains to analyse the case $p \in
\tilde{\eK}_1^c$.

Suppose $p \in \tilde{\eK}_1^c$. We start by considering an upper
bound for $\delta^i(p) = \delta^i$ defined in
\eqref{eqn:deltaEbardelta}. Let $\epsilon > 0$ be small enough such
that
\begin{equation}\label{eqn:positivity}
  u^i - \epsilon s^j > 0
  \quad\text{and}\quad
  u^i u^j - \epsilon^2 s^i s^j >  0.
\end{equation}
Such $\epsilon$ ensures that $\delta^i(p) > 0$. To get an upper
bound for $\delta^i(p)$ consider $u^i \leq \kappa-1$, $s^j > 1$,
$u^i u^j \geq 1$ and $s^i s^j \leq \kappa^2$. In
this case, for  $\epsilon > 0$,
\begin{equation}\label{eqn:deltaless}
  \delta^i(p)/\epsilon = \frac{u^i - \epsilon s^j}{u^i u^j -
    \epsilon^2 s^i s^j} < \frac{(\kappa-1) - 
    \epsilon}{1-\epsilon^2 \kappa^2}.
\end{equation}
For a sufficiently small $\epsilon > 0$, \eqref{eqn:deltaless} implies 
that $\delta^i(p)/\epsilon < \kappa$. As a consequence, for such an
$\epsilon$, we have that $\delta^i(p) < \kappa\epsilon$, and for $p
\in \tilde{\eK}_1^c$, the definition of $\ducal{H}$ in \eqref{eqn:H}
gives
\[
  \ducal{H}(p) = \sum_{v \in U(p)} p^i_v p^j_v = \sum_{v \in U(p)}
  \delta^i(p)\delta^j(p) < 
  \big|U(p)\big| 
  (\kappa\epsilon)^2 \leq \kappa(\kappa\epsilon)^2
\]
and the proof is concluded.
\end{proof}

\section{Proofs of the results for star graphs}\label{sec:star_proofs}
\begin{proof}[Proof of Theorem~\ref{th:star_G}]
Throughout let $\eta = \eta_v^i$ for all $v \in V$, $i \in
I_v$. Suppose first that $\epsilon = 0$.  We start by showing that the
set of fixed points of $\pi$ can be described as $\Fix(\pi)=\eS_1 \cup
\tilde{\eS}_1$, where $\eS_1$ defined as in the statement of the
theorem and
\[
  \tilde{\eS}_1 = \Big\{
  p \in \D \ \Big|\ p_1^i = 1\ \text{ for at least two } i \in [m]\
  \text{ while }\ p_1^j = 0\ \text{ for all } \,   p_1^j < 1 \Big\}.
%
\]
To show that any point $p\in \tilde{\eS}_1$ belongs to $\Fix(\pi)$, it
suffices to observe that $|\Supp{p^i}| = 1$ for all $i \in [m]$,
because in this case \eqref{eqn:THE_system} is tautologically true.
Now, we will show that any point $p \in \eS_1$ belongs to
$\Fix(\pi)$. Assume first that $p_1^i \in \{0, 1\}$ for all $i \in
[m]$. In this case $p$ belongs to $\Fix(\pi)$ because $|\Supp{p^i}| =
1$ for all $i \in [m]$. Assume now that $p_1^i = a$, for $a \in (0,
1)$ for some $i \in [m]$. The system \eqref{eqn:THE_system} gives
$\eta_1^i - \epsilon a = \eta_{i+1}^i -\epsilon(1-a)$. By hypothesis
we have that $\eta_1^i = \eta_{i+1}^i$ and
$\epsilon = 0$, and therefore $p \in \Fix(\pi)$ for all $a \in (0,1)$. To
conclude that $\Fix(\pi)=\eS_1\cup\tilde{\eS}_1$, suppose that $p
\notin \eS_1\cup\tilde{\eS}_1$. In this case, there are at least two
walks $i$, $k$ such that $p_1^i \in (0,1)$ and $a = p_1^k \in
(0,1)$. But in this case, for $v = 1$, $w = i+1$, $\rpar_v^{ii} = 0$
and $\rpar_v^{ij} = -1$ for $i \neq j$, \eqref{eqn:THE_system} gives
$\eta - a = \eta_v^i + \sum_{j \in I_v} \rpar_v^{ij} p_v^j = \eta_w^i
+ \sum_{j \in I_w} \rpar_w^{ij} p_w^j = \eta$, which contradicts $a
\in (0,1)$.

Assume that $p \in \tilde\eS_1$. For a suitable choice of
$i \in [m]$, we have that $p_1^i = 1$ and $\sum_{j\neq i}
p_1^j \geq 1$. As a consequence
\[
  \lim_{x \to p} \frac{\pi^i_{i+1}(x)}{x_{i+1}^i}
  =
  \frac{\eta^\alpha}{\big(\eta - \sum_{j\neq i}
  p_1^j\big)^\alpha}
  \geq
  \Big(\frac{\eta}{\eta - 1}\Big)^\alpha
  >
  1,
\]
A direct application of Theorem~\ref{th:no-toface} shows that $X$ has
zero probability of converging to any point $p \in \tilde{\eS}_1$. 

By Theorem~\ref{th:det}.$(iv)$, $\limset \subset \Fix(\pi)$ almost
surely. Now, as $\limset$ is connected and $\eS_1$ and $\tilde\eS_1$
are disconnected it holds that either $\limset \subset \eS_1$ or
$\limset \subset \tilde\eS_1$. The last inclusion can be discarded by
the fact that $\tilde\eS_1$ is finite and $X$ does not converges to
any of its points as shown in the previous paragraph. This concludes
the proof of the first assertion made by the theorem concerning the
case $\epsilon = 0$.

Let $\epsilon > 0$. By computing the solutions of
\eqref{eqn:THE_system} for each \( S \in \V \), it is readily verified
that the poits $p$ of \( \Fix(\pi) \) are of the following types:
\begin{enumerate}[label=$(b_{\arabic*})$, ref=$b_{\arabic*}$, nosep, 
  leftmargin=1cm]
\item\label{enu:b1} $p_1^i = 0$ for all $i \in [m]$.
\item\label{enu:b2} $p_1^i = 1$ for at least one $i \in [m]$, while
  $p_1^j \in \{0,1\}$ for all $j \neq i$. 
\item\label{enu:b3} $p_1^i = \epsilon/(|K| + 2 \epsilon - 1)$ for $i
  \in K$, while $p_1^j = 0$ for $j \notin K$. 
\item\label{enu:b4} $p_1^i = \epsilon/(m + 2 \epsilon - 1)$ for all $i 
  \in [m]$. 
\end{enumerate}

Assuming that $\epsilon > 0$ and applying Theorem \ref{th:no-toface}
shows that $X$ does not converges to any point of types \eqref{enu:b1}
and \eqref{enu:b2} because in both cases \eqref{eqn:abar} is satisfied
for suitable choices of $i \in [m]$ and $v \in V^i$.  Indeed, if $p$
is of type \eqref{enu:b1}, then $p^i_1 = 0$ and  $p^i_{i + 1} = 1$, and so 
\[
  \lim_{x \to p} \frac{\pi^i_1(x)}{x^i_1} =\frac{\big(\eta - \epsilon
    p_1^i - \sum_{j \neq i} p_1^j\big)^\alpha}{p^i_1 \big(\eta -
    \epsilon p_1^i - \sum_{j \neq i} p_1^j\big)^\alpha +  p^i_{i+1}
    \big(\eta - \epsilon p_{i+1}^i\big)^\alpha} =
  \bigg(\frac{\eta}{\eta - \epsilon}\bigg)^\alpha > 1. 
\]
If $p$ is of type \eqref{enu:b2}, then, for some $i \in [m]$,  we have
$p^i_1 = 1$ and  $p^i_{i + 1} = 0$, and so 
\begin{align*}
  \lim_{x \to p} \frac{\pi^i_{i+1}(x)}{x^i_{i+1}} 
  &=  
  \frac{\big(\eta - \epsilon p_{i+1}^i\big)^\alpha}{p^i_1 \big(\eta - 
  \epsilon p_1^i - \sum_{j \neq i} p_1^j\big)^\alpha +  p^i_{i+1}
  \big(\eta - \epsilon p_{i+1}^i\big)^\alpha} \\
  &=
  \bigg(\frac{\eta}{\eta
  - \epsilon - \sum_{j \neq i} p_1^j}\bigg)^\alpha  > 1.
\end{align*}

Assume now that $\epsilon \in \big(0, \frac12\big)$. The convergence
of $X$ towards the point $p$ of the form \eqref{enu:b4} can be ruled
out as follows. Relatively simple computations show that all the
eigenvalues of $D\pi(p)$ have real parts different from one. Moreover,
one of the eigenvalues of $D\pi(p)$ equals
\[
  \lambda(\epsilon) = \alpha \frac{\epsilon\big(m-1 +
    \epsilon\big)\big(2\epsilon -1\big)}{ 
  \big(m-1+2\epsilon\big)\big(-m(m-1) - (m+1)\epsilon +
  \epsilon^2\big)} +  1.
\]
It thus follows that $\lambda(\epsilon) > 1$ if $\epsilon \in \big(0,
\frac12\big)$, showing that $p$ is a linearly unstable point. By
observing that $p \in \iD$ and applying
Theorem~\ref{th:no-tointerior}, it follows that $X$ has zero
probability of converging to $p$. Only the set of points
\eqref{enu:b3} remain. These points are disconnected from those in
\eqref{enu:b1}, \eqref{enu:b2} and \eqref{enu:b4}. Moreover, observe
that the set of points in \eqref{enu:b3} is discrete and equals the
set $\eS_2$. As these points belong to $\Fix(\pi)$ and $\limset$ is
connected, items $(iii)$ and $(iv)$ of Theorem~\ref{th:det} show that
$X$ must converge a.s. towards any one point of $\eS_2$.

In order to show the last item of the theorem, assume that $\epsilon
\in \big(\frac12, 1\big)$. It suffices to show that $X$ does not
converge towards any point $p$ of the form \eqref{enu:b3}. This
follows from Theorem~\ref{th:no-toface}. Indeed, if $p$ is of type
\eqref{enu:b3}, there are $m-k$ indices $i$ with $p_1^i = 0$, and $k$
indices $j$ with $p_1^j = \epsilon/(k+2\epsilon-1)$ where $k=|K|$. So,
for any $i$ with $p_1^i = 0$, $p_{i+1}^i = 1$ and $\epsilon \in
\big(\frac12, 1\big)$, we obtain
\begin{align*}
  \lim_{x \to p} \frac{\pi^i_1(x)}{x^i_1}  
  &= 
  \frac{\big(
    \eta - \epsilon p^i_1 - \sum_{j \neq i} p^j_1\big)^\alpha
  }{
    p^i_1 \big( \eta - \epsilon p^i_1 - \sum_{j \neq i} p^j_1
    \big)^\alpha + p^i_{i+1} \big( \eta - \epsilon p^i_{i+1}
    \big)^\alpha
  } \\ 
  &= \bigg(\frac{
    \eta -  k \epsilon/(k+2\epsilon-1)
  }{
    \left(\eta - \epsilon \right)  \, \, 
  }\bigg)^\alpha > 1.
\end{align*}
As the point in \eqref{enu:b4} belongs to $\Fix(\pi)$ and is isolated
from the points in \eqref{enu:b1}-\eqref{enu:b3}, the result follows
from items $(iii)$ and $(iv)$ in Theorem~\ref{th:det}.
\end{proof}

\begin{proof}[Proof of Corollary~\ref{cor:star_overlap}]
Suppose first that  $\epsilon \in \big(0, 2/(m(m-1))\big)  \setminus
\{\frac12\}$. Then, by virtue of Theorem~\ref{th:star_G}, the process
\( X(n) \) converges to some point \( p \) such that 
\begin{equation}\label{eqn:partiacenter}
  p_1^i = \frac{\epsilon}{|K| + 2\epsilon - 1} \ \text{ for } i \in K,
  \quad \text{and} \quad p_1^j = 0 \ \text{ for } j \in [m] \setminus K 
\end{equation}
where $K$ is a non-empty subset of $[m]$, including $K = [m]$. When
$|K|=1$ it follows that $\ducal{H}(p)= \sum_{i<j} p^i_1p^j_1 = 0$ for
any $p$ satisfying \eqref{eqn:partiacenter}.

Suppose that $|K|\geq 2$. Since $\lim_n \ducal{H}(X(n)) =
\ducal{H}(p)$, it suffices to show that \(\ducal{H}(p) \leq \epsilon
\) for any point \( p \) satisfying \eqref{eqn:partiacenter}. To see
this, observe that \( p_1^i = 0 \) for all \( i \notin K \). As a
consequence, it follows that
\[
 \ducal{H}(p) = \frac12 \sum_{i,j\in K} p_1^i p_1^j 
 =
  \binom{k}{2} \left( \frac{\epsilon}{k + 2\epsilon -1}
 \right)^2,
\]
where  \( k = |K| \geq 2\). Since \( k  \in \{1, 2, \ldots, m\} \) and
$\epsilon \leq 2/(m(m-1))$, we have
\[
  \ducal{H}(p) =  
  \frac{\epsilon^2 k(k-1)}{2(k+ 2\epsilon-1)^2} 
  \leq
  \frac{\epsilon^2 k(k-1)}{2} 
  \leq
  \frac{\epsilon^2 m(m-1)}{2}
  \leq 
  \epsilon.
\]

It remains to analyse the case $\epsilon = 0$. When $\epsilon = 0$,
any point \( p \in \eS_1 \) satisfies \( \ducal{H}(p) = \sum_{i < j}
p_1^i p_1^j = 0 \). This implies that $\lim_n \ducal{H}(X(n)) =
\ducal{H}(p) = 0$, proving the second assertion and completing the
proof.
\end{proof}

\begin{proof}[Proof of Lemma~\ref{lem:star_with_preferences}] The
  proof proceeds along two main cases: $\eta >1$ and $0 < \eta < 1$.
 
\medskip 

\noindent\textbf{Case 1}: $\eta > 1$. In this case, $p = (p_1^1,
p_2^1, p_1^2, p_2^2) \in \Fix(\pi)$ if and only if $p_v^i \in \{0,
1\}$ for $i, v = 1, 2$.  To see this, suppose by contradiction and
without loss of generality that $p \in \Fix(\pi)$ with $p_v^1 \in
(0,1)$ for some $v \in \{1,2\}$. Since $p_1^1 + p_2^1 = 1$, it follows
that $p_1^1 > 0$ and $p_2^1 > 0$ and, from \eqref{eqn:THE_system} we
conclude that $\tilde \eta = \tilde \eta\ +\ \eta\ -\ p_1^2 \geq
\tilde \eta\ +\ \eta\ -\ 1 > \tilde \eta$. We now identify the limit
points to which $X$ may converge. For $p = (0, 1, a, 1-a)$, $a \in
\{0,1\}$, it follows that
\begin{align*}
  \lim_{x \to p} \frac{\pi_1^1(x)}{x_1^1}
  &=
  \frac{(\eta + \tilde{\eta}  - a)^\alpha}{\tilde{\eta}^\alpha} \geq 
  \Big(\frac{\eta + \tilde{\eta} - 1}{\tilde{\eta}}\Big)^\alpha > 1.
\end{align*}
Using the previous inequality and applying Theorem~\ref{th:no-toface}
shows that $X$ does not converge to points $p$ of these types. A
similar argument also shows that $X$ does not converge to points of
the form $p = (a, 1-a, 0, 1)$, $a \in \{0,1\}$. It then follows from
items $(iii)$ and $(iv)$ of Theorem~\ref{th:det} that $X$ converges
almost surely towards $(1, 0, 1, 0)$, as this point is disconnected
from the points $(0, 1, a, 1-a)$, $(a, 1-a, 0, 1)$, $a \in \{0, 1\}$,
and in fact the only element left in $\Fix(\pi)$.

\medskip

\noindent\textbf{Case 2}: $0 < \eta < 1$. The analysis in this case
proceeds along several subcases exhausting all possibilities. Suppose
first that $p = (1,0,1,0)$. In this case we have
\[
  \lim_{x \to p} \frac{\pi^1_2(x)}{x_2^1} =
  \Big(\frac{\tilde{\eta}}{\tilde{\eta} + \eta -1}\Big)^\alpha > 1,
\]
where the inequality follows because $\eta < 1$. The convergence of
$X$ to $p$ can now be discarded by using
Theorem~\ref{th:no-toface}. Next, assume that $p =
(0,1,0,1)$. Computations show that
\[
  \lim_{x \to p} \frac{\pi^1_1(x)}{x_1^1}
  =
  \Big(\frac{\tilde{\eta} + \eta}{\tilde\eta}\Big)^\alpha > 1.
\]
A direct application of Theorem~\ref{th:no-toface} shows that $X$ does
not converges to $p$.

Let $p = (a, 1-a, b, 1-b)$, $a \in (0,1)$ and $b \in \{0, 1\}$. A
verification of \eqref{eqn:THE_system} implies that $\tilde \eta =
\tilde \eta + \eta - b$, where $b = 0$ or $b = 1$. But this
contradicts the assumption that $0 < \eta < 1$. This shows that $p
\notin \Fix(\pi)$. Arguing in the same way, we conclude that any $p =
(a, 1-a, b, 1-b)$, with $a \in \{0, 1\}$ and $b \in (0,1)$, does not
belong to $\Fix(\pi)$ either.

Now, let $p = (a, 1-a, b, 1-b)$, where $a \in (0,1)$ and $b \in
(0,1)$. A simple verification using \eqref{eqn:THE_system} shows that
$p \in \Fix(\pi)$ iff $a = b = \eta$. The spectrum of the Jacobian
matrix $D\pi$ at $p$ in this case equals
\[
  \frac{-\eta + \eta^2}{\tilde\eta} + 1, \quad
  0,\quad  0,\quad
  \frac{\eta - \eta^2}{\tilde\eta} + 1.
\]
As $(\eta-\eta^2)/\tilde\eta$ is positive for all $0 < \eta < 1$,
and therefore $(\eta-\eta^2)/\tilde\eta + 1 > 1$ for all $0 < \eta
< 1$, we conclude by Theorem~\ref{th:no-tointerior} that $X$ does not
converge to $p$.

The remaining elements in $\Fix(\pi)$ when $0 < \eta < 1$ are the
points $(1,0,0,1)$ and $(0,1,1,0)$. These points are themselves
disconnected and also disconnected from the points of $\Fix(\pi)$
considered previously.  Items ($iii$) and ($iv$) in
Theorem~\ref{th:det} show therefore that $\big\{\lim_n X(n) =
(1,0,0,1)\big\} \bigcup \big\{\lim_n X(n) = (0,1,1,0)\big\}$ occurs
almost surely.
\end{proof}

\section{Proofs of the results for cyclic
  graphs}\label{sec:cycle_proofs}
\begin{lemma}[Fixed points for $\pi$ in cyclic graphs, $\epsilon = 0$]
\label{lem:Equilibria_Cyclic_G} Let $C_m$ be a cyclic graph with $m$
edges and for each $i \in [m]$ let $W^i$ be a random walk defined on
$G^i$ according to the transition probabilities in \eqref{eqn:pi}. 
Then
$
  \Fix(\pi) = \bigcup_{r=1}^4 \eC_r,
$
where $\eC_r$, $r=1,2,3,4$ are defined before the statement of
Theorem~\ref{th:cyclic_G}.
\end{lemma}

\begin{proof}
As $G$ is cyclic, to simplify the exposition of the arguments
throughout, we will use the indices $1$ and $m + 1$ interchangeably in
the sense that $1$ may be viewed as the successor of $m$, and
similarly for $m$ and $0$, as $m$ may be viewed as the predecessor of
$1$.  Accordingly, we may refer to the coordinate $x_1^m$ of $x \in
\D$ as $x_{m + 1}^m$ or $x_{1}^0$.

Given this convention, for any cycle of length $m\in \N$, assume that
the edge starting at vertex $v \in V$ is mixed, that is, the
limit occupation measure for the walk $W^v$ equals $p_{v}^{v} =
a$ and $p_{v+1}^{v} = 1-a$ for some $a \in (0,1)$.  According to
\eqref{eqn:rhoij} and \eqref{eqn:etaiv} we have that $\rpar_{v}^{v-1,v} =
\rpar_{v+1}^{v,v+1} = -1$ and thus  from \eqref{eqn:THE_system} 
that 
$p_{v}^{v-1} = p_{v+1}^{v+1}$. As this holds for any $v = 1, 2,
\ldots m$, it follows that $p \in \Fix(\pi)$, where $p_v^i = \frac12$
for $v \in V$ and $i \in I_v$. That is, $\eC_1 \subset
\Fix(\pi)$. More generally, for $p_{v}^{v} = a \in (0,1)$, we can
choose $p_{v}^{v-1} = p_{v+1}^{v+1} = b$ for any $b \in
(0,1)$. Successive applications of \eqref{eqn:THE_system} show that
the sequence of edges $\big(\!\!\dbr{a, 1-a}$, $\dbr{b, 1-b}$,
$\dbr{1-a, a}$, $\dbr{1-b, b}\!\!\big)$, or any multiple of this
sequence, can be closed to form a cycle proving thus that $\eC_2$ is
formed by points in $\Fix(\pi)$. This shows that $\eC_1 \cup \eC_2
\subset \Fix(\pi)$.

Suppose now that $p\in \eC_3$. In this case, any edge is  unmixed and so
any point in $\eC_3$ satisfies the system \eqref{eqn:THE_system}.
This shows that $\eC_3 \subset \Fix(\pi)$.

Let $p \in \eC_4$. Without loss of generality, suppose that the edge
starting at $v \in V$ is mixed. Since, by definition of $\eC_4$, we
have $p_{v}^{v-1} = 1 = p_{v+1}^{v+1}$ or $p_{v}^{v-1} = 0 =
p_{v+1}^{v+1}$, we have that $p$ satisfies the system
\eqref{eqn:THE_system}.  This shows that $\eC_4 \subset \Fix(\pi)$.

It remains to show that the reverse inclusion holds, namely $\Fix(\pi)
\subset \eC_1 \cup \eC_2 \cup \eC_3 \cup \eC_4$. Suppose $p \in
\Fix(\pi)$. According to the previous considerations, we have the
following: if all the edges of $p$ are mixed then it follows by
\eqref{eqn:THE_system} that $p \in \eC_1 \cup \eC_2$, and if all the
edges of $p$ are unmixed then by \eqref{eqn:THE_system} we have that
$p \in \eC_3$. If all mixed edges of $p$ are flanked by unmixed edges
then it also follows from \eqref{eqn:THE_system} that $p \in
\eC_4$. To conclude we show that, if $p\in\Fix(\pi)$, then either all
the edges of $p$ are mixed, or all the edges of $p$ are unmixed, or
each mixed edge of $p$ is flanked by unmixed edges. Suppose this is
not the case, then $p$ has two consecutive mixed edges followed by an
unmixed edge of the form $\dbr{a, 1-a}$, $\dbr{b, 1-b}$, $\dbr{c,
1-c}$, where $a, b \in (0,1)$ and $c \in \{0,1\}$. Since $p \in
\Fix(\pi)$, it follows by equation system \eqref{eqn:THE_system}, that
$1 - a = c$. But this is impossible as $a \in (0,1)$ and $c \in
\{0,1\}$. This shows that $\Fix(\pi) = \bigcup_{r=1}^4 \eC_r$ and
concludes the proof.
\end{proof}

\begin{proof}[Proof of Theorem~\ref{th:cyclic_G}] 
Proof of $i)$. By Theorem~\ref{th:det}.$(iv)$, $\limset \subset
\Fix(\pi)$ almost surely. The first assertion of the theorem follows
from the previous inclusion because, by
Lemma~\ref{lem:Equilibria_Cyclic_G}, $\Fix(\pi) = \bigcup_{r=1}^4
\eC_r$.

Proof of $ii)$. We will now analyse the particular case when $m$ is
not a multiple of four. In this case observe that $\eC_2 =
\varnothing$ implying that $\eC_1$ is disconnected from
$\eC_3\cup\eC_4$.  Since by Theorem~\ref{th:det}.(iii), $\limset$ is
connected, it follows that either $\limset \subset \eC_1$ or $\limset
\subset \eC_3\cup\eC_4$. Since $\eC_1$ is a singleton, the inclusion
$\limset \subset \eC_1$ implies that $X$ converges to the unique point
of $\eC_1$. However, as we show in the next paragraph, $X$ does not
converge to such a point when $m$ is not a multiple of four, and
therefore we conclude that $\limset \subset \eC_3\cup\eC_4$.

Let $m>3$ be an integer that is not a multiple of four. Let $p \in
\eC_1$ and write $\eta = \eta_v^i$ for all $i \in I_v$, $v \in
V$. Straightforward computations show that the Jacobian matrix of the
underlying vector field at $p$ has all eigenvalues with real part
different from zero and one eigenvalue equal to
$\alpha/(2\eta-1)$. This shows that $p$ is linearly unstable. When
$m=3$, Descartes' rule of signs shows that the characteristic
polynomial of the Jacobian at $p$ has a real positive root. The
convergence of $X$ to a point in $\eC_1$ can thus be discarded for any
$m$ that is not a multiple of four by using
Theorem~\ref{th:no-tointerior}.
\end{proof}

\begin{proof}[Proof of Corollary~\ref{cor:if_X_converges}]
Throughout the proof, let $\eta = \eta_v^i$ for all $i \in I_v$,  $v
\in V$. Observe that $\eta > 1$ because of \eqref{eqn:eta}. 

A description of the points in $\Fix(\pi)$ to which $X$ can converge
to proceeds by identifying the elements in the sets $\eC_r$, $r = 1,
2, 3, 4$, of Lemma~\ref{lem:Equilibria_Cyclic_G} that can be discarded
via Theorem~\ref{th:no-tointerior} or Theorem~\ref{th:no-toface}. Let
$p \in \bigcup_{r=1}^4 \eC_r$ be a limit point of $X$. We will show
that $p \in \tilde{\eC}_1 \cup \tilde{\eC}_3 \cup \tilde{\eC}_4$.

Suppose first that $p \in \eC_1$. In this case $m$ is a multiple of
four, and by the definition of $\tilde \eC_1$, we have that $p \in
\eC_1 = \tilde \eC_1$. Indeed, if $m$ is not a multiple of four, then,
as was observed in the last paragraph of the proof of
Theorem~\ref{th:cyclic_G}, $p$ cannot be a limit point of $X$.

Suppose now that $p \in \eC_2$. For any $a, b \in (0,1)$, $\alpha >
0$, the analysis of the Jacobian matrix of the vector field at the
point $p \in \eC_2$ shows that this equilibrium is linearly
unstable. As $p\in \D^\circ$, convergence to this point can be ruled
out via Theorem~\ref{th:no-tointerior}.

Assume that $p \in \eC_3 \cup \eC_4$.  Assume, by contradiction, that
$X$ converges to a point $p \in (\eC_3 \cup \eC_4) {\setminus} (\tilde
\eC_3 \cup \tilde \eC_4$). In that case, $p$ has a subsequence either
of type
\begin{equation}
 \label{eqn:sequenceA}
 \ldots, \dbr{1-a, a}, \dbr{0,1}, \dbr{1,0}, \ldots\ \text{ for }\  a
 \in [0,1)   
\end{equation}
or  
\begin{equation}
 \label{eqn:sequenceB}
  \ldots, \dbr{1, 0}, \dbr{0,1}, \dbr{b,1-b}, \ldots\ \text{  
for } b \in (0,1]. 
\end{equation}

Observe that the condition that defines $\tilde\eC_4$ via
\eqref{eqn:C4tilde} includes the condition in the definition of
$\tilde\eC_3$ by \eqref{eqn:C3tilde}, as $a$ and $b$ may be $0$ and
$1$ respectively. If $p$ contains the subsequence
\eqref{eqn:sequenceA}, then, for a given vertex $v\in V$, we have
$p_v^{v-1}=a$, $p_v^v=0$, $p_{v+1}^v = 1$ and $p_{v+1}^{v+1}=1$, where
$a \in [0,1)$. This situation is depicted by the following scheme

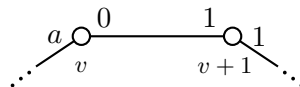
{\centering
  \begin{tikzpicture}[line width=.8pt]
    \draw[white] (-2,-1) -- (4,-1) -- (4,.75) -- (-2,.75);
    \draw[fill=black] (-.69,-.48) circle (.3pt);
    \draw[fill=black] (-.80,-.56) circle (.3pt);
    \draw[fill=black] (-.91,-.65) circle (.3pt);
    \draw (-0.6,-0.4) -- (0,0) -- (2,0) -- (2.6,-0.4);
    \draw[fill=black] (2.69,-.48) circle (.3pt);
    \draw[fill=black] (2.80,-.56) circle (.3pt);
    \draw[fill=black] (2.91,-.65) circle (.3pt);
    \draw[fill=white] (0,0) circle (3.25pt);
    \draw[fill=white] (2,0) circle (3.25pt);
    \node at (-0.35,0.0) {\small $a$};
    \node at (0.3,0.25) {\small $0$};
    \node at (1.7,0.25) {\small $1$};
    \node at (2.35,0.0) {\small $1$};
    \node at (0.0,-0.4) {\scriptsize $v$};
    \node at (1.9,-0.4) {\scriptsize $v+1$};
\end{tikzpicture}
\par}
\noindent In this case $\lim_{x\to p} \pi_{v}^v(x)/x_{v}^v =
\big((\eta-a)/(\eta-1)\big)^\alpha > 1$ as $a <
1$. Theorem~\ref{th:no-toface} allows then to discard the convergence
of $X$ towards any point $p$ of this form.

If $p$ contains the subsequence \eqref{eqn:sequenceB}, then, for a
given vertex $v\in V$, $p_v^{v-1}=0$, $p_v^v=0$, $p_{v+1}^v = 1$ and
$p_{v+1}^{v+1}=b$. This situation is depicted by the following scheme

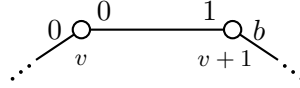
{\centering
  \begin{tikzpicture}[line width=.8pt]
    \draw[white] (-2,-1) -- (4,-1) -- (4,.75) -- (-2,.75);
    \draw[fill=black] (-.69,-.48) circle (.3pt);
    \draw[fill=black] (-.80,-.56) circle (.3pt);
    \draw[fill=black] (-.91,-.65) circle (.3pt);
    \draw (-0.6,-0.4) -- (0,0) -- (2,0) -- (2.6,-0.4);
    \draw[fill=black] (2.69,-.48) circle (.3pt);
    \draw[fill=black] (2.80,-.56) circle (.3pt);
    \draw[fill=black] (2.91,-.65) circle (.3pt);
    \draw[fill=white] (0,0) circle (3.25pt);
    \draw[fill=white] (2,0) circle (3.25pt);
    \node at (-0.35,0.0) {\small $0$};
    \node at (0.3,0.25) {\small $0$};
    \node at (1.7,0.25) {\small $1$};
    \node at (2.35,0.0) {\small $b$};
    \node at (0.0,-0.4) {\scriptsize $v$};
    \node at (1.9,-0.4) {\scriptsize $v+1$};
\end{tikzpicture}
\par}
\noindent In this case, $\lim_{x\to p} \pi_v^v(x)/x_v^v =
\big(\eta/(\eta - b)\big)^\alpha >1$ as $b \in (0,1]$ and $\eta >
1$. Theorem~\ref{th:no-toface} allows then to discard the convergence
of $X$ towards any point $p$ of this form. This contradicts the fact
that $X$ converges to a point $p$ having a sub-sequence of the form
\eqref{eqn:sequenceA} or \eqref{eqn:sequenceB}, and so $X$ converges
to a point $p$ in $\tilde \eC_3 \cup \tilde \eC_4$.
\end{proof}

\begin{proof}[Proof of Theorem~\ref{th:cyclic_G_e}]
The proof consists in showing that there is a sufficiently small
$\epsilon_* > 0$ such that, for $0 < \epsilon < \epsilon_*$,
we have that $|\D(S)| = 1$ for at least one $S\in \V$, and $|\D(S)|
\leq 1$ for all other $S\in \V$, in which case, by item $(i)$ of
Theorem~\ref{th:det}, we have $1 \leq |\Fix(\pi)| \leq |\V| < \infty$
for $0 < \epsilon < \epsilon_*$. By Theorem~\ref{th:det}.$(iv)$
it then follows that $X$ converges almost surely to some element of 
$\Fix(\pi)$.

Assume first that $|S_i| = 1$ for all $i \in [m]$. In this case, it
follows trivially that $|\D(S)| = 1$. Indeed, in this case, the only
one element $x$ of $\D(S)$ is such that $x_v^i = 1$ iff $v \in S^i$
and $x_v^i = 0$, otherwise, $i = 1, 2, \ldots m$.  Now, since $\V$ is
finite, it is sufficient to show that, for each $S\in \V$, where
$|S^i| > 1$ for at least one $i = 1$, there is a sufficiently small
$\epsilon(S) > 0$ such that $|\D(S)| \leq 1$ for $0 < \epsilon <
\epsilon(S)$. Since $|S^i| \leq 2$ for all $i \in [m]$, the proof
follows by considering the following two types of $S \in \V$:
\begin{enumerate}[nosep,label=(\Roman*),ref=(\Roman*)]
\item\label{itm:I} $|S_i| = 2$ for all $i \in [m]$; and  
\item\label{itm:II} $|S_i| = 1$ and $|S_j| = 2$ for at least two
different $i,j \in [m]$.
\end{enumerate}

Since $V^m = \{m, 1\}$ and $V^i = \{i, i+1 \}$ for $i =1, 2, \ldots
m-1$, the variables in \eqref{eqn:THE_system} are such that
$x_m^m + x_{1}^m = 1$ and
$x_i^i + x_{i+1}^i = 1$ for $i = 1, 2, \dots, m - 1$. Substituting
$x_{1}^m$ by $1 - x_m^m$ and $x_{i+1}^i$, by $1 - x_i^i$ for $i = 1,
2, \dots, m - 1$, the system \eqref{eqn:THE_system} reduces to a
system with $m$ variables $z_1 = x_1^1, z_2 = x_2^2, \ldots, z_m =
x_m^m$ such that
\begin{equation}\label{eqs2}
\left\{
\begin{gathered}
  z_{i} = C_i \quad\text{for}\quad |S^i| = 1,\\
  -z_{[i-1]_m}  +2 \varepsilon z_{{[i]}_m}  -z_{{[i+1]}_m} = -1
  \quad\text{for}\quad |S^i| = 2,
\end{gathered}
\right. \qquad i \in [m]
\end{equation}
where $[i]_m = m$ if $i$ is a multiple of $m$ and $[i]_m = i \bmod m$
if $i$ is not a multiple of $m$, and $C_i = C_i(S)$ is such that $C_i
= 1$ if $i \in S^i$, and $C_i = 0$, otherwise.

In case~\ref{itm:I}, the system \eqref{eqs2} is given by the equations
$-z_{[i-1]_m}  +2 \varepsilon z_{{[i]}_m}  -z_{{[i+1]}_m} = -1$, $i \in
[m]$, only. The corresponding matrix of coefficients of
this system is the $m$-by-$m$ circulant matrix with first row equal to
$(2\epsilon, -1, 0, \ldots, 0, -1)$. This matrix has the following
characteristic polynomial 
\[
  \mathcal P_S(\epsilon) =
  2^m \prod_{k=0}^{m-1} \Big(\epsilon - \cos\Big(\frac{2\pi}{m}
  k\Big)\Big).
\]
It is evident from the form of $\mathcal P_S(\epsilon)$ that $\mathcal 
P_S(\epsilon) \neq 0$ and hence $|\D(S)| \leq 1$ for $0 < \epsilon < 
\epsilon(S)$, where $\epsilon(S)$ is a sufficiently small positive number 
depending on $S$. 

In case~\ref{itm:II}, we observe that the system \eqref{eqs2} is
equivalent to a new system having the same variables $z_i$, $i \in
[m]$ as in \eqref{eqs2}, with the only difference that, in the $i$-th
equation, $-z_{[i-1]_m} +2 \varepsilon z_{{[i]}_m} - z_{{[i+1]}_m} =
-1$, presented in the second line of \eqref{eqs2}, the variable
$z_{k}$ is replaced by the constant $C_{k}$ whenever $|S^k| = 1$, $k
\in \{[i-1]_m, [i+1]_m\}$, $i \in [m]$, $|S^i| = 2$.

Let $K = \{z_i \mid i \in [m], \,\, |S^i| = 2 \}$. Without loss of
generality, we can partition $K$ into $L$ subsets of consecutively
indexed variables, such that $K = \bigcup_{j = 1}^L \{z_{\ell_j + 1},
z_{\ell_j + 2}, \ldots, z_{\ell_j + r_j}\}$, where $0 < \ell_j + r_j <
\ell_{j+1} < m$, $\ell_j \geq 1$ and $r_j \geq 1$, $j = 1,2, \ldots
L$. Setting $C_{j,1} = C_{\ell_j}$, $C_{j,2} = C_{\ell_j + r_j +1}$,
and $z_{j,k} = z_{\ell_j + k}$, where $k = 1, 2, \ldots, r_j$ and $j =
1, 2, \ldots, L$, the equations concerning the variables in $K$ can be
written as $L$ independent subsystems of linear equations, where the
$j$-th subsystem is of the form
\begin{equation}\label{eqn:subsys}
  \left\{
    \begin{aligned}
       &2 \varepsilon\, z_{j,1} - z_{j,2} =
      -1 + C_{j,1}                                                 \\  
       - & z_{j,k} +2 \varepsilon\, z_{j,k+1} - z_{j,k+2} = -1, 
       \qquad k = 1, \ldots, r_j - 2,                              \\  
      - & z_{j, r_j-1}  +2 \varepsilon\, z_{j, r_j} =
      -1 + C_{j,2}
    \end{aligned}
  \right.
\end{equation}

Now, the corresponding matrix of coefficients of subsystem
\eqref{eqn:subsys}, hereafter denoted by $T_j$, is a $r_j \times r_j$
symmetric tridiagonal T\"oeplitz matrix with band elements $-1,
2\epsilon, -1$. According to Theorem 2.4 in \cite{BA05}, $T_j$ has
characteristic polynomial 
\[
  \mathcal P_{S,j}(\epsilon)
   =
   2^{r_j} \prod_{k=1}^{r_j} \Big(\epsilon -
   \cos\Big(\frac{\pi}{r_j+1} k\Big)\Big).
\]
The form of $\mathcal P_{S,j}(\epsilon)$ implies that $\mathcal
P_{S,j}(\epsilon) \neq 0$ for $0<\epsilon<\epsilon_j(S)$, where
$\epsilon_j(S)$ is a sufficiently small positive number depending on
$S$ and $j$. Observe now that for any $S\in\V$, the matrix of
coefficients in case \ref{itm:II} can be represented, without loss of
generality, as a block diagonal matrix of the form
\[
 \begin{bmatrix}
  I_1 &     &        &      &      &         &      &      \\
      & T_1 &        &      &      &         &      &      \\
      &     & \ddots &      &      &         &      &      \\
      &     &        & I_j  &      &         &      &      \\
      &     &        &      & T_j  &         &      &      \\
      &     &        &      &      & \ddots  &      &      \\
      &     &        &      &      &         & I_L  &      \\
      &     &        &      &      &         &      & T_L
 \end{bmatrix}
\]
where $I_j$, $j=1, \ldots, L$, are identity matrices of size $\ell_1
\times \ell_{1}$ for $j = 1$, and $(\ell_{j} - (\ell_{j-1} +
r_{j-1}))\times (\ell_j - (\ell_{j-1} + r_{j-1}))$ for $j > 1$. The
coefficient matrix of the full system has therefore characteristic
polynomial $\mathcal P_S(\epsilon) = \prod_{j=1}^L \mathcal
P_{S,j}(\epsilon)$. As a consequence, for any $S$ in the
case~\ref{itm:II} we have that $\mathcal P_S(\epsilon) \neq 0$ and
hence $|D(S)| \leq 1$ for $0 < \epsilon < \epsilon(S)$, where
$\epsilon(S) = \min_j \epsilon_j(S)$. This concludes the proof of the
theorem.
\end{proof}

\section*{Appendix}
\appendix
\section{Proofs of Lemmas~\ref{lem:SA} and \ref{lem:limitset}} 
\begin{proof}[Proof of Lemma~\ref{lem:SA}] In order to show
\eqref{eqn:SA}, observe first that for any $i \in [m]$ and $v \in
V^i$,
\begin{align*}
  X_{v}^i(n+1) - X_{v}^i(n)
  &=
  \frac{1+\sum_{k=0}^{n-1}\xi_{v}^i(k) + \xi_{v}^i(n)}{d^i+n+1}
  -
    \frac{1+\sum_{k=0}^{n-1}\xi_{v}^i(k)}{d^i+n} \\
  &=
    \frac{1}{d^i+n+1}\Big(-X^i_{v}(n) + \xi_{v}^i(n)\Big).
\end{align*}
This gives
\begin{align*}
  X(n+1) - X(n)
  &=
  \Big\{
  \Big(-X(n)+\E[\xi(n)\mid\F_n]\Big) + \Big(\xi(n) - \E[\xi(n)\mid  
  \F_n]\Big)\Big\}\iG_n   \\
  &=
    \big(
    -X(n)+\E[\xi(n)\mid\F_n] + U(n)
    \big)\iG_n            \\
  &=
    \big(
    -X(n)+\pi(X(n)) + U(n)
    \big)\iG_n.
\end{align*}
The last equality follows by observing that $\E[\xi^i_v(n) \mid
\F_n] = \Prb(W^i(n+1) = v \mid \F_n) = \pi_v^i(X(n))$. Equation
\eqref{eqn:SA} follows from this by using the definition of $F$
given in \eqref{eqn:THE_field}.
\end{proof}

\begin{proof}[Proof of Lemma~\ref{lem:limitset}]
Let $X$, $U$, $\gamma_n^i$, and $\Xi_n$ be as in the statement of
Lemma~\ref{lem:SA}. For $d = \sum_{i=1}^m d^i$, let $\widehat
\Xi_n$ be the $d\times d$ diagonal matrix with all diagonal elements
equal to $\widehat\gamma_n$ where $\widehat\gamma_n =
\max_{i}\gamma_n^i$. Both assertions about the limit set $\limset$
follow from Theorem 1.2 in \cite{B96} provided the following
conditions are satisfied:
\begin{enumerate}[nosep]
\item[(\emph{i})] $X=\{X(n); n\geq 0\}$ is bounded;
\item[(\emph{ii})] $\widehat\gamma_n$ satisfies:
\[
  \lim_{n\to\infty} \widehat\gamma_n = 0,
  \quad
  \sum_{n\geq 0}\widehat\gamma_n = \infty,
  \quad\text{and}\quad
  \sum_{n\geq 0} \widehat\gamma_n^2 < \infty;
\]
\item[(\emph{iii})] For each $T>0$, almost surely it holds
  that
\begin{equation*}
  \lim_{n\to\infty}\bigg(\sup_{r: 0 \leq \tau_r - \tau_n \leq
    T}                
  \bigg\| \sum_{k=n}^{r-1} U(k) \widehat\Xi_k\bigg\|\bigg) = 0,
\end{equation*}
where $\tau_0=0$ and $\tau_n = \sum_{k=0}^{n-1}\widehat\gamma_k$.
\end{enumerate}
Conditions (\emph{i}) and (\emph{ii}) are immediate. To show
(\emph{iii}) define $M(n) = \sum_{k=0}^n U(k)\widehat\Xi_k$,
$n\geq0$. It is simple to verify that $\{M(n); n\geq 0\}$ is a
martingale with respect to $\{\F_{n+1}; n\geq 0\}$;
further
\begin{align*}             
  \sum_{n\geq 0} \E\Big[\big\| M(n+1) - M(n)\big\|^2 \, \Big|\,
  \F_{n+1}\Big]
  &=
  \sum_{n\geq 0} \E\Big[\Big\|U(n+1)\widehat\Xi_{n+1}\Big\|^2\,\Big|\,
    \F_{n+1}\Big]                                             \\ 
  &\leq 
  \sum_{n\geq 0} \bigg(\sum_{i=1}^m \widehat\gamma_{n+1}
    d^i\bigg)^2  \\
  &\leq 
    d^2 \sum_{n\geq 0} \widehat\gamma_{n+1}^2 < \infty.
\end{align*}
A standard result about square integrable martingales then shows that
$M(n)$ converges almost surely to a finite limit and thus $\{M(n);
n\geq 0\}$ is almost surely a Cauchy sequence.  This is sufficient to
establish condition (\emph{iii}).
\end{proof}

\section{Proof of Lemmas~\ref{Coupling-Working} and \ref{lem:b.r.v}}
\begin{proof}[Proof of Lemma~\ref{Coupling-Working}]
Let $a > 1$ be the value of the limit in \eqref{eqn:abar}. Choose
$\epsilon > 0$ such that $\pi_k^i(x) / x^i_k > (a+1)/2$ for all $x \in
\txcal{N}_\epsilon(p)$, which is an $\epsilon$-neighborhood of $x$ in
$\D$.  In this case $\pi_k^i(X(m)) \geq X^i_k(m) (a+1)/2$, whenever
$X(m) \in \txcal{N}_\epsilon(p)$. Since $X(n)$ converges to $p$ with
positive probability, we can choose $n_0 > 0$ sufficiently large such that
$A = \{X(m) \in \txcal{N}_\epsilon(p)\ \text{for all } m \geq n_0\}$
occurs with positive probability and $\Prb\big(\pi_k^i(X(m)) \geq
X^i_k(m) (a+1)/2 \mid A\big)$ for all $m \geq n_0$.

Now, set $q = d^i$, define
\[
  b(n) = \frac{(a-1)/2 - (q/n  + \delta)}{q/n + 1 + \delta},
\]  
and suppose that $\delta > 0$ is sufficiently small and $n_0 > 0$ is
sufficiently large such that  
\[
  b = b(n_0) > 0. 
\]
From the definition of $b$ and the fact that $n > n_0$, it follows
that
\begin{equation}
  \label{eqn:geq1b}
  \frac{n(a + 1)/2}{q + \floor{n(1+\delta)}} 
  \geq
  \frac{n(a + 1)/2}{q + n(1+\delta)} = 1 +  b(n) \geq 1 + b(n_0)  = 1
  +  b. 
\end{equation}
Assuming that $X^i_k(m) \geq n X^i_k(n)/(q + \floor{n(1+\delta)})$ for
all $m \in T_{n,\delta}$ and $n > n_0$, conditionally on $A$ it
follows with probability one that
\begin{align*}
  \Prb(W^i_{m+1}=k\mid \F_m) 
  =
  \pi^i_k(X(m))
  &\geq  
    X^i_k(m) \frac{a+1}{2} \\
  &\geq
    X_k^i(n)\frac{n(a+1)/2}{q + \floor{n(1+\delta)}} \,\, \geq \,
    X^i_k(n)(1+b).
\end{align*}
The last inequality follows from \eqref{eqn:geq1b}.
 
It remains to show that $X^i_k(m) \geq n X^i_k(n)/(q +
\floor{n(1+\delta)})$ for all $n > n_0$ and all $m \in T_{n,\delta}$.
Fix $n$ and $m$ arbitrarily such that $n_0 < n$ and $m \in
T_{n,\delta}$.  Then $0< n_0 < n \leq m \leq \floor{n(1+\delta)} -
1$. From the definition of $X_k^i(n)$ in
\eqref{eqn:occupation_measure} and the previous inequalities, we
conclude that
\begin{align*}
 X^i_k(m)
    &\geq
      \frac{1}{q+m}\bigg(1+\sum_{\ell=1}^n \Ind\big\{W^i(\ell) =
      k\big\} \bigg)     
  =
  \frac{q+n}{q+m} X^i_k(n) \\
    &\geq \frac{n}{q + 1 +m} X^i_k(n)     \\
    &\geq \frac{n}{q + 1 + (\floor{n(1+\delta)}  - 1)} 
      X^i_k(n). 
\end{align*}
This completes the proof of the lemma.
\end{proof}

\begin{proof}[Proof of Lemma~\ref{lem:b.r.v}]
To simplify notation, we first reidex the random variables
$\Ber(\ell)$ as follows. Let $B(1) = 1$, $B(\ell) = 0$ for $\ell = 2,
3,\ldots, q$, $B(\ell) = \Ber(\ell - q)$ for all $\ell > q$, and $\bar
B(n) = \frac{1}{n}\big(\sum_{\ell=1}^n B(\ell) \big)$ for $n >
q$. Note that $\bar B(n) = \bar \Ber(n - q)$ for $n > q$. It is now
sufficient to show the following equivalent assertion rewritten in
terms of $B$ and $\bar B$ instead of $\Ber$ and $\bar \Ber$. There are
no $\delta >0$, $b>0$, $q > 0$, $n_0 > q$ and $A \in \F$ with
$\Prb(A)>0$ such that
\begin{equation}
 \label{eqn:b1bber2}
 \Prb \Big\{\Prb \big(B(m+1) = 1 \mid \F_m \big) \geq \bar
  B(n)(1+b)  \,  \Big  |  \, A \Big\}  = 1
  \ \ \text{for all }\ m \in   T_{n,\delta} \ \ \text{and}\ \ n >
  n_0.
\end{equation}

Suppose by contradiction there are $\delta >0$, $b>0$, $q > 0$, $n_0 >
q$, and $A \in \F$, with $\Prb (A) > 0$, such that \eqref{eqn:b1bber2}
holds. Define $n_{\tau+1} = \floor{n_{\tau}(1+\delta)}$ 
for $\tau = 0, 1, 2, \ldots$ and set $\bar B_\tau = \bar B(n_\tau)$,
$S_\tau = S(n_\tau) = \sum_{\ell = 1}^{n_\tau} B(\ell)$ and
$D_{\tau+1} = D(n_{\tau+1}) = S(n_{\tau + 1}) -S(n_{\tau}) =
\sum_{\ell = n_\tau + 1}^{n_{\tau+1}} B(\ell)$. Then for $\tau \geq
0$, we have that
\begin{align*}
  \bar B_{\tau + 1}
  &=
    \frac{S \big (\floor{n_\tau (1+
    \delta)} \big )}{\floor{n_\tau (1+
    \delta)} }  \\
  &\geq  
    \frac{\delta}{1 + \delta} \frac{S \big (\floor{n_\tau (1+
    \delta)} \big )}{\delta n_\tau}                                 \\ 
  &= 
    \frac{\delta}{1 + \delta} \bigg [S(\floor{n_\tau (1+
    \delta)}) - S(n_\tau) \bigg ] \frac{1}{\delta n_\tau} +
    \frac{\bar B(n_\tau)}{1 + \delta}                               \\
  &= 
    \Big ( \frac{\delta}{1 + \delta} \frac{D_{\tau+1}}{\delta
    S_\tau} + \frac{1}{1 + \delta} \Big )   \bar B_\tau          \\
  &= 
    F(Z_{\tau+1}) \bar B_\tau, 		
\end{align*}
where $F(y) = \delta/(1 + \delta) y + 1/(1 + \delta)$ and
$Z_{\tau+1} = D_{\tau+1}/(\delta S_\tau)$.

As a consequence, it follows that 
\[
 \E\big[\log \big(\bar B_{\tau + 1}\big)   \Ind_A\big] 
 \geq
 \E\big[\log \big (\bar B_{\tau}) \Ind_A\big] +
 \E\big[\log \big(F(Z_{\tau+1})\big) \Ind_A\big].
\]	
Since $\bar B_{\tau} \geq 1/n_\tau$, we have that $\E\big[\log \big
(\bar B_{\tau} \big ) \Ind_A \ \big] > - \infty$ for all $\tau \geq
0$. To reach a contradiction to the assumption that $\Prb(A)>0$, we
will use the definition of $A$ presented in \eqref{eqn:b1bber2} and
show that $\lim_{\tau \to \infty} \E\big[\log\big(F(Z_{\tau+1})\big)
\Ind_A \ \big] > 0$ which then leads to $\lim_{\tau \to \infty}
\E\big[\log \big (\bar B_{\tau}) \Ind_A \ \big] = \infty$. The last
equality contradicts the fact that $\log \big (\bar B_{\tau} \big )
\leq 0$ for all $\tau = 0, 1, 2, \ldots$.

Since $\E\big[\log \big(F(Z_{\tau+1})\big) \Ind_A\big] = \E\big[\log
F(Z_{\tau+1}) \ \big|\ A\big]\Prb(A)$ and $\Prb(A)>0$, it is
sufficient to show that
\[
  \lim_{\tau \to \infty} \E\big[\log F(Z_{\tau+1}) \ \big|\ A
  \big] > 0.
\]
Set $a_\tau = \big [(n_{\tau + 1} - n_\tau)/(\delta n_\tau) \big ] (1
+ b) - b/2$ and $\epsilon_{\tau} = \Prb (Z_{\tau+1} < a_\tau \, | \, A
\, )$. Taking into account that $Z_{\tau+1} \geq 0$, it follows that
$Z_{\tau+1} \geq Y_{\tau+1} = a_\tau \Ind{\{Z_{\tau+1} \geq a_\tau \}}
\in \{0, a_\tau\}$. Since $\log$ and $F$ are increasing, then
\begin{align*}
  \E\big[\log F(Z_{\tau+1}) \, \big| \, A\big] 
  \geq  
   \E\big[\log F(Y_{\tau+1})\, \big| \,  A\big]
   %
 =
    \epsilon_{\tau} \log F(0)  +  (1 -  \epsilon_{\tau})\log F(
    a_{\tau}).
\end{align*}
Now, observing that $a_\tau \to (1 + b/2)$ and assuming that
$\epsilon_\tau \to 0$ gives
\[
  \lim_{\tau \to \infty}
  \E\big[\log F(Z_{\tau+1}) \ \big|\ A \big] \geq \log F(
  1 + b/2) > 0.
\]

We will show now that $\epsilon_\tau \to 0$. Let $U_\ell$, $\ell
= 1, 2, 3,\ldots$ be independent uniformly distributed random
variables over $[0,1]$ and set $\tilde D_{\tau+1} = \sum_{\ell =
n_{\tau} + 1}^{n_{\tau + 1}} \Ind\{U_\ell < \bar B_\tau (1 + b)\}$ and
$\tilde Z_{\tau+1} = \tilde D_{\tau+1}/(\delta S_\tau)$. The
desired limit for $\epsilon_\tau$ may be obtained  from the following
facts:
\begin{enumerate}[$(i)$.,nosep,leftmargin=1cm]
\item $\E\big[\Prb\big(\tilde Z_{\tau+1} < a_\tau \mid \F_{n_\tau}A
  \big)\, \big|\, A \  \big]
  \leq  \E \big[K_\tau/S_\tau\mid A \big]$, where 
\item  $K_\tau$ is a sequence of numbers such that $\lim_{\tau
\to \infty} K_\tau \in [0, \infty)$;
\item $\Prb \big (\lim_{\tau \to \infty} S_\tau = \infty \, \big | \,
A \ \big ) = 1$;
\item $\E\big[\Prb \big (Z_{\tau+1} < a_\tau \, | \, \F_{n_\tau}A \big
) \, \big | \, A \big] \leq \E\big[\Prb \big (\tilde Z_{\tau+1} <
a_\tau \, | \, \F_{n_\tau}A \big ) \, \big | \, A \big]$;
\end{enumerate}
as in this case 
\begin{equation} \label{eqn:converg.epsilon}
  \epsilon_{\tau} = \E\Big[\Prb \big (Z_{\tau+1} < a_\tau  \, |
  \, \F_{n_\tau}A  \big ) \, \Big | \, A \  \Big]
  \leq  \E \big[K_\tau/S_\tau   \,  \big | \,
  A \big] \to 0. 
\end{equation}

We proceed now to show each of these facts.

\textbf{Fact $(i)$}.
Set $\mu_\tau = \big [(n_{\tau + 1} - n_\tau)/(\delta n_\tau) \big ]
(1 + b)$ and note that $\E\big[\tilde Z_{\tau + 1} \mid \F_{n_\tau}
A \big] = \mu_\tau$ and $a_\tau = \mu_\tau - \frac{b}{2}$.
Chebyshev's inequality gives $\Prb\big(\tilde Z_{\tau+1} < a_\tau
\, | \, \F_{n_\tau} A \big ) = \Prb \big (\tilde Z_{\tau+1} <
\mu_\tau - b/2 \, | \F_{n_\tau} A \big ) \leq \Var(\tilde Z_{\tau +
1} | \, \F_{n_\tau} A ) /(b/2)^2$. Since $\tilde Z_{\tau+1} =
\tilde D_{\tau+1}/(\delta S_\tau)$ and $\tilde D_{\tau+1}$ has
a binomial probability distribution given $\F_{n_\tau} A$, it follows
that
\[
\Var(\tilde Z_{\tau+1} | \, \F_{n_\tau} A ) = \Var(\tilde
D_{\tau+1} | \, \F_{n_\tau} A)/(\delta S_\tau)^2 \leq \E(\tilde
D_{\tau+1} | \, \F_{n_\tau} A)/(\delta S_\tau)^2 = \mu_\tau
/(\delta S_\tau).
\]
Letting $K_\tau = \mu_\tau / \big [\delta (b/2)^2
\big ]$ and combining the two previous inequalities gives
\begin{equation}\label{eqn.chebyn}
   \Prb \big (\tilde Z_{\tau+1} < a_\tau   \, |
  \, \F_{n_\tau} A \big ) \leq  K_\tau/S_\tau.
\end{equation}
To conclude take the conditional expectation given $A$ on
both sides of \eqref{eqn.chebyn}. 

\textbf{Fact $(ii)$}. 
To see that $\lim_{\tau \to \infty} K_\tau \in [0, \infty)$, note that
$K_\tau = \mu_\tau / \big [\delta (b/2)^2 \big ]$, $\mu_\tau = (b + 1)
(n_{\tau + 1} - n_\tau)/(\delta n_\tau)$, and $(n_{\tau + 1} -
n_\tau)/(\delta n_\tau) = (\floor{n_\tau (1+ \delta)} -
n_\tau)/(\delta n_\tau) \to 1$ as $\tau \to \infty$.
  
\textbf{Fact $(iii)$}.
Let $E_m = \{B(m) = 1\}$ for $m \geq n_0$. Since $S_\tau = \sum_{m =
1}^{n_\tau} B(m)$, it is sufficient to show that $\Prb(E_m\ \
\text{i.o.} \, | \, A) = 1$.  Fix $\ell\geq n_0$ arbitrarily, set
$p_\ell= \Prb\big( E_\ell\, \big | \, A \big )$ and, for $m \geq
\ell$, and define $p_{m+1} = \Prb\big( E_{m+1} \, \big | \,
\bigcap_{k= \ell}^{m} E_{k}^c \cap A \big )$. For $N > \ell$, we have
that
\[
\Prb \bigg(\bigcap_{m = \ell}^N  E_m^c \,\Big|\, A \  \bigg) =
    \prod_{m = \ell}^{N-1} \Big (1 - p_{m+1} \Big) (1 - p_\ell)
    \leq                                                              
    \exp\bigg(- \sum_{m = \ell}^{N-1}  p_{m+1}\bigg) 
\]
and thus, assuming that $p_{m+1} \geq 1/m$ for $m \geq \ell\geq n_0$
gives
\[
  \lim_{N \to \infty}    \Prb \bigg(\bigcap_{m = \ell}^N  E_m^c \,
  \Big| \, A \ 
    \bigg) 
    \leq
    \lim_{N \to \infty} \exp\bigg(- \sum_{m = \ell}^{N-1}
    \frac{1}{m}\bigg) = 0  
\]
which then shows that $\Prb(E_m\ \ \text{i.o.}\mid A \ ) = 1$.

We now verify that $p_{m+1} \geq 1/m$ for $m \geq \ell\geq n_0$. Chose
$\tau \geq 0$, such that  $m \in T_{n_\tau, \delta}$. It follows from
\eqref{eqn:b1bber2} that
\begin{align*}
  p_{m+1}
  &=
  \E\bigg[\Prb \big ( B(m+1) = 1 \, | \, \F_{m} \big )
  \ \bigg|\ \bigcap_{k=\ell}^{m} E_{k}^c \cap A \ \bigg] \\
  &\geq
  \E\bigg[\bar B(n_\tau)(b + 1) \ \bigg|\
    \bigcap_{k=\ell}^{m} E_{k}^c \cap A \ \bigg] \\
  &\geq \frac{1}{m}, 
\end{align*}
where the previous inequality holds because
\[
  \bar B(n_\tau)(b + 1)
  \geq \sum_{k=1}^{n_\tau} \frac{B(k)}{n_\tau} \geq
  \frac{B(1)}{n_\tau} =  \frac{1}{n_\tau}
  \geq \frac{1}{m}.
\]
Note that the last inequality holds because $m \in T_{n_\tau, \delta}$
and therefore $m \geq n_\tau$.

\textbf{Fact $(iv)$}.
Since $Z_{\tau + 1} = D_{\tau+1} /(\delta S_\tau)$ and $\tilde Z_{\tau
+ 1} = \tilde D_{\tau+1} /(\delta S_\tau)$, and since the conditional
expectation given $A$ is monotone it is sufficient to show that
\begin{equation}
\label{eqn:dominance0}
 \Prb \big (D_{\tau+1} \geq  a \mid \F_{n_\tau} A \big )
 \geq
 \Prb\big(\tilde D_{\tau + 1}\geq   a  \mid \F_{n_\tau}  A \big)
 \quad \text{ for all }\ a \geq 0.
\end{equation}
Note that $D_{\tau+1}$ and $\tilde D_{\tau+1}$ are respectively
the following sums of Bernoulli random variables $\sum_{m \in
T_{n_\tau, \delta}} B(m +1)$ and $\sum_{m \in T_{n_\tau, \delta}}
B'(m+1)$, where $B'(m+1) = \Ind\{U_{m+1} < \bar B(n_\tau) (1 +
b)\}$. While the latter sum considers independent and identically
distributed random variables given $\F_{n_\tau}A$, the former
considers possibly dependent ones. Nevertheless, by
\eqref{eqn:b1bber2}, conditioned on $\F_{n_\tau} A$, it follows 
that
\begin{equation}
\label{eqn:dominance2}
 \Prb\big(B(m+1) = 1  \mid \F_{m} A  \big) \geq  B(n_\tau)(1 + b)
 =
 \Prb \big(B'(m+1) = 1 \mid \F_{n_\tau} A \big)
\end{equation}
holds with probability 1, for any $m \in T_{n_\tau, \delta}$. It is
relatively simple to verify now that \eqref{eqn:dominance0} follows
from \eqref{eqn:dominance2} and a standard coupling argument.  This
concludes the proof of the lemma.
\end{proof}

\textbf{Acknowledgement}
We like to thank Adriano J. Holanda for writing the code used in
various simulations, made while the work of this article was carried
out.

\bibliographystyle{amsplain} 

\providecommand{\bysame}{\leavevmode\hbox to3em{\hrulefill}\thinspace}
\providecommand{\MR}{\relax\ifhmode\unskip\space\fi MR }
\providecommand{\MRhref}[2]{%
  \href{http://www.ams.org/mathscinet-getitem?mr=#1}{#2}
}
\providecommand{\href}[2]{#2}

\end{document}